\documentclass[10pt]{article}
\usepackage[a4paper, margin=1in]{geometry} 
\usepackage[utf8]{inputenc} 
\usepackage[T1]{fontenc}    
\usepackage{cite}           
\usepackage{color}          
\usepackage{hyperref}
\usepackage{amsmath, amssymb, amsfonts}
\usepackage{stmaryrd}       
\usepackage{bbold}          
\usepackage{mathrsfs}       
\usepackage[version=4]{mhchem} 
\usepackage{authblk}
\usepackage{subcaption}     
\usepackage{tikz}           
\usetikzlibrary{arrows}     
\usepackage{flowchart}      
\usepackage{algorithm}          
\usepackage{algpseudocode}      

\usepackage{authblk}        
\usepackage{amsthm}         
\usepackage[utf8]{inputenc}
\usepackage[T1]{fontenc}
\usepackage{amsmath}
\usepackage{amsfonts}
\usepackage{amssymb}
\usepackage[version=4]{mhchem}
\usepackage{stmaryrd}
\usepackage{hyperref}
\hypersetup{colorlinks=true, linkcolor=blue, filecolor=magenta, urlcolor=cyan,}
\usepackage{caption}

\usepackage{algorithm}
\usepackage{algcompatible} 

\newtheorem{theorem}{Theorem}[section]
\newtheorem{definition}{Definition}[section]
\newtheorem{lemma}{Lemma}[section]
\newtheorem{proposition}{Proposition}[section]
\newtheorem{example}{Example}[section]
\newtheorem{remark}{Remark}[section]

\numberwithin{equation}{section} 

\DeclareMathOperator*{\argmin}{arg\,min}
\makeatletter
\renewcommand{\maketitle}{
  \begin{center}
    {\LARGE\bfseries \@title \par}
    \vskip 1em
    {\normalsize
      \@author
    }
  \end{center}
}
\makeatother

\title{FINITE ELEMENT ANALYSIS OF NASH EQUILIBRIUM OF  BI-OBJECTIVE OPTIMAL CONTROL
PROBLEM GOVERNED BY STOKES EQUATIONS WITH
$L^2$-NORM STATE-CONSTRAINTS}

\author[1,2]{Kedarnath Buda \\ \texttt{kedarnath.buda@gmail.com}}
\author[1]{B.V. Rathish Kumar\thanks{*Corresponding author} \\ \texttt{bvrk.paper.1707@gmail.com}}
\author[1]{Anil Rathi \\ \texttt{anil.rathi.1707@gmail.com}} 

\affil[1]{Department of Mathematics and Statistics\\ Indian Institute of Technology Kanpur, India}
\affil[2]{Department of Mathematics\\ Government College, Sundargarh, Odisha, India}

\date{}

\begin{document}

\maketitle

\begin{abstract}



This paper investigates the Nash equilibrium of a bi-objective optimal control problem governed by the Stokes equations. A multi-objective Nash strategy is formulated, and fundamental theoretical results are established, including the existence, uniqueness, and analytical characterization of the equilibrium. A finite element framework is developed to approximate the coupled optimal control system, and the corresponding optimality conditions for both the continuous and discrete formulations are rigorously derived and analyzed. Furthermore, \textit{a priori} finite element error estimates are obtained for the discrete problem, ensuring convergence and stability of the proposed method. The theoretical results are corroborated by numerical experiments, which demonstrate the accuracy and computational efficiency of the finite element approach.

\end{abstract}

\vspace{1em}

\section{Introduction}

Optimal control of partial differential equations (PDEs) is a fundamental topic in applied mathematics, engineering, and physics, with wide-ranging applications in fluid dynamics, heat transfer, and material design. Classical studies focus on controlling systems governed by elliptic, parabolic, and hyperbolic PDEs to achieve desired states while minimizing a cost functional~\cite{lions1971optimal}.
Optimal control theory governed by partial differential equations (PDEs) is an essential approach for tackling a diverse range of modeling and simulation challenges in engineering, physics, and economics. This field centers on finding control variables that minimize a designated cost functional, subject to PDEs embodying the underlying physical processes. Addressing these problems involves combining functional analysis with computational algorithms specifically designed for large-scale systems. Of the many PDE-constrained optimization problems, those posed by the Stokes or Navier--Stokes equations are particularly significant in fluid dynamics. The Stokes equations capture the steady motion of incompressible viscous fluids, offering a fundamental and readily analyzable model for applications such as drag reduction, flow control, and transport phenomena. The pursuit of optimal control for the Stokes system is a well-studied area, both theoretically and numerically; see, for instance, \cite{lions1971optimal, temam1979, gunzburger1989, casas1986, BergouniouxKunisch1997}. Core analytical concerns include guaranteeing the existence, uniqueness, and regularity of optimal solutions, as well as formulating necessary optimality conditions in variational form.

In recent years, there has been growing interest in studying PDE control problems that involve multiple goals or players, all interacting within the same system. Instead of just optimizing for one objective, these situations are better described using ideas from game theory, such as the Nash equilibrium. Here, each participant tries to minimize their own cost, and the Nash equilibrium ensures that no one can do better by changing their strategy alone. The foundation for these types of problems comes from game theory and variational inequality theory \cite{aubin1981, rockafellar1970, clarke1983}, with their use in PDE-related problems starting in \cite{ramos1998a, ramos1998b, gunzburger1992}. Since then, understanding Nash equilibrium in distributed control has become an important area in computational optimal control.

The finite element method (FEM) provides a convenient  and powerful technique for the numerical approximation of PDE-constrained optimization problems, also in complex geometries. The method allows for systematic discretization of the state, adjoint, and control equations using conforming finite-dimensional subspaces. Classical results on finite element approximation of elliptic and Stokes systems can be found in \cite{brenner2008mathematical, ciarlet1978, girault1996}.

The development of numerical schemes, particularly finite element methods, has enabled accurate and efficient approximations of optimal control problems, offering both theoretical convergence guarantees and practical computational solutions. These advances have played a central role in the numerical analysis of partial differential equation (PDE)–constrained optimization, where reliable discretization techniques are crucial for establishing stability and convergence of the computed controls.

The concept of the Nash equilibrium has been introduced in the context of multi-objective and game-theoretic optimal control problems, where multiple decision-makers (players) simultaneously optimize their own objectives subject to shared system dynamics. A Nash equilibrium represents a state in which no player can unilaterally improve their performance, thus providing a natural setting for bi-objective or multi-agent control systems~\cite{rahman2015fem, ramos2023nash, StackelbergNashStokes}. While classical studies have addressed the existence, uniqueness, and computation of Nash equilibria for PDE-constrained problems, investigations involving \emph{state-constrained} formulations remain limited—particularly when such constraints are imposed in an $L^{2}$-norm sense.

In this work, we investigate bi-objective optimal control problems governed by the Stokes equations with $L^{2}$-norm state constraints. We establish the existence and uniqueness of Nash equilibria for the continuous formulation and subsequently construct a finite element approximation that preserves the structure of the underlying PDE system together with the imposed control constraints. Both the exact and discretized optimality systems are rigorously derived, providing a consistent analytical framework for the error analysis. Furthermore, we derive \emph{a priori} error estimates for the finite element approximation, demonstrating optimal-order convergence in both the $L^{2}$- and $H^{1}$-norms. These theoretical findings are validated through a series of numerical experiments, which confirm the accuracy and efficiency of the proposed scheme and highlight its applicability in computing Nash equilibrium for complex bi-objective control configurations.

The combination of rigorous mathematical analysis and computational validation offers a comprehensive framework for the finite element analysis of Nash equilibria in Stokes-based bi-objective control problems, thereby extending the current literature to encompass state-constrained scenarios. Our study closely follows the analytical structure of~\cite{niu2011stokes} in the finite element analysis of Stokes control problems, but it advances the theory by establishing new \emph{a priori} error estimates for the Nash equilibrium formulation while providing an analytical justification of existence and uniqueness. This generalization introduces a theoretical contribution by addressing the coupled interaction between multiple control objectives under the Stokes flow setting.

The present work is devoted to the finite element analysis of a bi-objective distributed optimal control problem governed by the stationary Stokes equations under $L^{2}$-norm state constraints. The problem is formulated as a Nash equilibrium between two players, each influencing the same flow field but targeting distinct desired states. The state constraints are imposed in integral form, introducing additional Lagrange multipliers into the optimality system. The analysis includes the derivation of the variational formulation of the coupled state–adjoint–control system, together with \textit{a priori} error estimates for the corresponding finite element approximation.

The main contributions of this paper are threefold. 
First, in Section~\ref{sec:2}, we formulate a bi-objective distributed optimal control problem for the Stokes equations, interpreted as a Nash equilibrium system with integral $L^{2}$-norm constraints on the state. This section also establishes the associated first-order optimality conditions through the introduction and characterization of the coupled state and adjoint equations. 
Second, in Section~\ref{sec:3}, we develop a conforming finite element discretization of the continuous problem and derive rigorous \textit{a priori} error estimates for the discrete optimal solutions. 
Third, in Section~\ref{sec:4}, we present a series of numerical experiments that validate the theoretical results, illustrating the accuracy, convergence, and stability of the proposed computational scheme. 
Finally, Section~\ref{sec:5} summarizes the conclusions and outlines possible extensions of the present work.

\section{Problem Formulation: Bi-objective Optimal Control of Stokes PDE}\label{sec:2}


Let us define some standard definitions and notations that we are going to use in our analysis, as in \cite{adams2003sobolev}.  Let $\Omega$ be an open, bounded and connected domain in $\mathbb{R}^{d}$ for $d=2$ or 3. Let $\boldsymbol{\alpha}$ be a multi-index such that $\boldsymbol{\alpha} = (\alpha_{1}, \ldots, \alpha_{d})$ and $|\boldsymbol{\alpha}| = \alpha_{1} + \ldots + \alpha_{d}$. Define the conventional Sobolev space 
$$
\mathbf{W}^{m,p} := \left\{ \mathbf{y} \in \mathbf{L}^{p} : \frac{\partial^{|\boldsymbol{\alpha}|} \mathbf{y}}{\partial x_{1}^{\alpha_{1}} \ldots \partial x_{d}^{\alpha_{d}}} \in \mathbf{L}^{p}(\Omega), ~~ \forall ~ |\boldsymbol{\alpha}| \leq m \right \},
$$ 
with respect to the norm 
$$
\|y\|_{W^{m,p}} = \left(\sum_{|\alpha| \leq m} \left\| \frac{\partial^{|\boldsymbol{\alpha}|}{y}}{\partial x_{1}^{\alpha_{1}} \ldots \partial x_{d}^{\alpha_{d}}} \right\|^{p} \right)^{\frac{1}{p}}.
$$
In particular, denote $H^{m}(\Omega) = W^{m,2}(\Omega)$ and $L^{p}(\Omega) = W^{0,p}(\Omega)$. Also, define 
$$
L_{0}^{2}(\Omega) : = \left\{q \in L^{2}(\Omega) : \int_{\Omega} q = 0 \right\} ~~ \text{and} ~~ H_{0}^{1}(\Omega) : = \left\{ y \in H^{1}(\Omega) : y = 0 \text{ on } \partial \Omega \right\}.
$$
If $\mathbf{v} = (v_{1}, v_{2}, \dots, v_{d}) \in \left(W^{m,p}(\Omega) \right)^{d}$, 
$$
\|\boldsymbol{v}\|_{\mathbf{W}^{m,p}}^{p} = \|v_{1}\|_{W^{m,p}}^{p} + \|v_{2}\|_{W^{m,p}}^{p} + \cdots + \|v_{d}\|_{W^{m,p}}^{p}.
$$

Define $\mathbf{L}^{p}(\Omega)=\left(L^{p}(\Omega)\right)^{d}, \mathbf{H}^{m}(\Omega)=$ $\left(H^{m}(\Omega)\right)^{d}$ and $\mathbf{W}^{m, p}(\Omega)=\left(W^{m, p}(\Omega)\right)^{d}$ the usual vector-valued Sobolev spaces with norms $\|\cdot\|_{m ; \Omega}=\|\cdot\|_{H^{m}(\Omega)}$ and $\|\cdot\|_{m, p ; \Omega}=\|\cdot\|_{W^{m, p}(\Omega)}$, respectively. Denote $(\cdot, \cdot)$ to be the inner product defined on the bounded and open set $\Omega$.
$$
\mathbf{U}=\mathbf{L}^{2}(\Omega) \times \mathbf{L}^{2}(\Omega), \quad \mathbf{H}=\left(H_{0}^{1}(\Omega)\right)^{d}, \quad Q=L_{0}^{2}(\Omega) 
$$
which stand for the control space, the velocity state space, and the pressure state space, respectively.

\subsection{Optimal control problem for the Stokes System}
We begin by formulating the classical optimal control problem along with its corresponding weak formulation. For a fixed control function $\mathbf{v}_2$, let $\alpha_1 > 0$ be a given regularization parameter, and define the objective functional $\mathcal{J}_1: \mathbf{L}^{2}(\Omega) \to \mathbb{R} $ as follows:
$$
\mathcal{J}_1(\mathbf{y}, \mathbf{v}_1)=\frac{1}{2} \int_{\Omega}\left|\mathbf{y}-\mathbf{y}_{1,d}\right|^{2}+\frac{\alpha_i}{2} \int_{\Omega}\left|\mathbf{v_1}\right|^{2}.
$$
We investigate the following state-constrained optimal control problem for a fixed control $\mathbf{v_2}\in \mathbf{U}_2$:
\begin{equation}\label{cost}
\min_{\mathbf{v_1}\in \mathbf{U}_1} \mathcal{J}_1(\mathbf{y}(\mathbf{v_1}), \mathbf{v_1})
\end{equation}
subject to the Stokes equations:
\begin{equation}\label{state}
\begin{cases}
-\nu \Delta \mathbf{y}(\mathbf{v}_1))+\nabla p(\mathbf{v}_1)) = \mathbf{f} + B_1 \mathbf{v}_1 + B_2 \mathbf{v}_2 & \text { in } \Omega,  \\ 
\nabla \cdot \mathbf{y}(\mathbf{v}_1)) = 0 & \text { in } \Omega, \\ 
\mathbf{y}(\mathbf{v}_1) = \mathbf{0} & \text { on } \partial \Omega,
\end{cases}
\end{equation}
where the constant $\nu > 0$ denotes the viscosity cofficient and $B_i : \mathbf{L}^{2}(\Omega) \rightarrow \mathbf{L}^{2}(\Omega)$ is a continuous linear operator for $i=1,2$. The spatial domain \( \Omega \subset \mathbb{R}^N \) is bounded with Lipschitz boundary. The target (desired) state is represented by $\mathbf{y}_{1,d}\in \mathbf{L}^{2}(\Omega)$, and \( \mathbf{f} \in\mathbf{L}^{2}(\Omega) \) denotes a given source term. The weak form of the problem is obtained by defining the following bilinear forms:
$$
(\mathbf{w}, \mathbf{z})=\sum_{i=1}^{d} \int_{\Omega} w_{i} z_{i}, \quad a(\mathbf{w}, \mathbf{z})=\nu \sum_{i=1}^{d}\left(\nabla w_{i}, \nabla z_{i}\right), \quad b(\mathbf{z}, q)=-(q, \nabla \cdot \mathbf{z}).
$$

The bi-linear form $a(\cdot, \cdot)$ is both continuous and coercive in $\mathbf{H}$; hence, there exist constants $a_{l}>0, a_{u}>0$ such that
\begin{equation}\label{inequality}
a_{l}\|\mathbf{z}\|_{1 ; \Omega}^{2} \leq a(\mathbf{z}, \mathbf{z}), \quad|a(\mathbf{w}, \mathbf{z})| \leq a_{u}\|\mathbf{w}\|_{1 ; \Omega}\|\mathbf{z}\|_{1 ; \Omega}, \quad \forall \mathbf{w}, \mathbf{z} \in \mathbf{H}. 
\end{equation}

Furthermore, from the well-known results in \cite{girault1996, temam1979}, it follows that the bilinear form $b(\cdot, \cdot)$ satisfies the Ladyzhenskaya–Babu{\v s}ka–Brezzi (LBB) stability condition and is continuous on the corresponding spaces. Hence, there exist constants $b_{l}>0$ and $b_{u}>0$ such that
\begin{equation}
b_{l} \leq \inf _{q \in Q} \sup _{\mathbf{z} \in \mathbf{H}} \frac{b(\mathbf{z}, q)}{\|\mathbf{z}\|_{1 ; \Omega}\|q\|_{0 ; \Omega}}, \quad b(\mathbf{z}, q) \leq b_{u}\|\mathbf{z}\|_{1 ; \Omega}\|q\|_{0 ; \Omega}, \quad \forall \mathbf{z} \in \mathbf{H}, q \in Q . 
\end{equation}

Therefore, the weak formulation of the optimal control problem defined in \ref{cost} takes the following form:
\begin{equation}\label{optimal control problem-1}
\min _{\mathbf{y}(\mathbf{v}_1) \in \mathbf{H}} \mathcal{J}_1(\mathbf{y}(\mathbf{v}_1), \mathbf{v}_1) 
\end{equation}
subject to
$$
\begin{cases}
a(\mathbf{y}(\mathbf{v}_1), \mathbf{w}) + b(\mathbf{w}, p(\mathbf{v}_1)) =(\mathbf{f} + B_1 \mathbf{v}_1 + B_2 \mathbf{v}_2, \mathbf{w}) & \forall \mathbf{w} \in \mathbf{H}, \\ 
b(\mathbf{y}(\mathbf{v}_1), q) = 0 & \forall q \in Q.
\end{cases}
$$
Analogously, by fixing the control variable $\mathbf{v}_1$, we can formulate the associated optimal control problem as follows:
\begin{equation}\label{optimal_control_proble-2}
\min_{\mathbf{y}(\mathbf{v}_2) \in \mathbf{H}} \, \mathcal{J}_2(\mathbf{y}(\mathbf{v}_2), \mathbf{v}_2)
\end{equation}
subject to the equations:
\begin{equation} \label{state2}
\begin{cases}
-\nu \Delta \mathbf{y}(\mathbf{v}_2) + \nabla p(\mathbf{v}_2) = \mathbf{f} + B_1 \mathbf{v}_1 + B_2 \mathbf{v}_2 & \text { in } \Omega,  \\ 
\nabla \cdot \mathbf{y}(\mathbf{v}_2) = 0 & \text { in } \Omega, \\ 
\mathbf{y}(\mathbf{v}_2) = \mathbf{0} & \text { on } \partial \Omega.
\end{cases}
\end{equation}

\subsection{Bi-Objective Nash Control Problem for the Stokes System}

The existence and uniqueness of the solution to the problem can be established by the usual method (see \cite{lions1971optimal}). \\
\textbf{Cost Functionals and Admissible Sets} \\
For each player $i \in \{1,2\}$, the individual cost functional is defined by
\begin{equation}
\mathcal{J}_i(\mathbf{y}, \mathbf{v}_1, \mathbf{v}_2) := \frac{1}{2} \left( \|\mathbf{y} - \mathbf{y}_{i,d}\|_{0,\Omega}^2 + \alpha _i \| \mathbf{v}_i\|_{0,\Omega}^2 \right) dt,   \label{eq:cost_functional}
\end{equation}
where $\mathbf{y}_{i,d} \in\mathbf{L}^{2}(\Omega)$ is the desired state for player $i$, and $\alpha_i > 0$ is a  regularization parameter. The admissible control sets are denoted by $\mathbf{U}_i = \mathbf{L}^{2}(\Omega)$, and we define the joint admissible set as
$$
\mathbf{U} := \mathbf{U}_1 \times \mathbf{U}_2.
$$
\textbf{Notation:} From now onwards, we denote the reduced cost functionals by
\begin{equation}\label{eq:reduced_cost}
    \mathcal{J}_i(\mathbf{v}_1, \mathbf{v}_2) := \mathcal{J}_i(\mathbf{y}(\mathbf{v}_1, \mathbf{v}_2), \mathbf{v}_1, \mathbf{v}_2), \quad i=1,2.
\end{equation}
In this section, we formulate a bi-objective optimal control problem constrained by the steady Stokes equations, under a Nash equilibrium study. We seek a control pair $(\mathbf{v}_1, \mathbf{v}_2) \in \mathbf{U}$ for which the corresponding two-player Nash optimality conditions hold:

\noindent The optimization problem for each player is defined by minimizing their individual cost functional $\mathcal{J}_i(\mathbf{v}_1, \mathbf{v}_2)$, 
\begin{enumerate}
    \item \textbf{Player 1's Problem:}
    Player 1 seeks the control $\mathbf{v}_1 \in \mathbf{U}_1$ that minimizes the functional $\mathcal{J}_1$, given that $\mathbf{v}_2$ is fixed:
    \begin{equation}\label{eq:cost_player1}
    \min_{\mathbf{v}_1 \in \mathbf{U}_1} \mathcal{J}_1(\mathbf{v}_1, \mathbf{v}_2) = \frac{1}{2} \|\mathbf{y}(\mathbf{v}_1, \mathbf{v}_2) - \mathbf{y}_{1,d}\|_{0, \Omega}^2 + \frac{\alpha_1}{2} \|\mathbf{v}_1\|_{0, \Omega}^2 .
    \end{equation}

    \item \textbf{Player 2's Problem:}
    Player 2 seeks the control $\mathbf{v}_2 \in \mathbf{U}_2$ that minimizes the functional $\mathcal{J}_2$, given that $\mathbf{v}_1$ is fixed:
    \begin{equation}\label{eq:cost_player2}
    \min_{\mathbf{v}_2 \in \mathbf{U}_2} \mathcal{J}_2(\mathbf{v}_1, \mathbf{v}_2) = \frac{1}{2} \|\mathbf{y}(\mathbf{v}_1, \mathbf{v}_2) - \mathbf{y}_{2,d}\|_{0, \Omega}^2 + \frac{\alpha_2}{2} \|\mathbf{v}_2\|_{0, \Omega}^2.
    \end{equation}
\end{enumerate}
Here, $\mathbf{y}_{i,d}$ is the desired state for player $i$, $\alpha_i > 0$ is the regularization parameter for the control cost, and $U_{i}$ is the admissible control set for player $i$.

Consider the map 
$$
\begin{aligned}
\mathcal{J}_1 : & ~ \mathbf{L}^{2}(\Omega) \to \mathbf{L}^{2}( \Omega) ~~~~~ \text{defined as}  \\
& ~ \mathbf{v}_1 \longmapsto \mathcal{J}_1(\mathbf{v}_1,\mathbf{v}_2)
\end{aligned}
$$
and 
$$
\begin{aligned}
\mathcal{J}_2 : & ~ \mathbf{L}^{2}( \Omega) \to \mathbf{L}^{2}( \Omega) ~~~~~ \text{defined as}  \\
& ~ \mathbf{v}_2 \longmapsto \mathcal{J}_2(\mathbf{v}_1, \mathbf{v}_2)
\end{aligned}
$$
both $\mathcal{J}_1$ and $\mathcal{J}_2$ are convex functionals for fixed $\mathbf{v}_2$ and $\mathbf{v}_1$, respectively.

Let us consider the bi-objective optimal control problem associated with the Stokes equations. Let 
$$
(\mathbf{v}_1, \mathbf{v}_2) \in \mathbf{U}_1 \times \mathbf{U}_2
$$ 
denote the control pair, and let $\mathbf{y}$ denote the corresponding state governed by the Stokes system
\begin{equation}\label{eq:stokes_state_system}
\begin{cases}
 - \nu \Delta \mathbf{y} + \nabla p = \mathbf{f} + B_1 \mathbf{v}_1 + B_2 \mathbf{v}_2, & \text{in } \Omega , \\
\nabla .\mathbf{y} = 0, & \text{in } \Omega , \\
\mathbf{y}(\mathbf{v_1,v_2})=\mathbf{0}, & \text { on } \partial \Omega.
\end{cases}
\end{equation}
The above formulation describes a non-cooperative bi-objective optimization problem, in which each control $\mathbf{v}_j$ seeks to steer the common state $\mathbf{y}$ toward its respective target $\mathbf{y}_{j,d}$, while accounting for the influence of the other player's control. The Nash equilibrium concept is employed to identify a pair of admissible controls $(\mathbf{v}_1, \mathbf{v}_2)$ such that no player can unilaterally decrease their own cost functional. 

\subsubsection{Nash Equilibrium Definition}

\begin{definition}[Nash Equilibrium]
A feasible control pair $\mathbf{u} = (\mathbf{u}_1, \mathbf{u}_2) \in \mathbf{U}$ is called a \textbf{Nash solution} or \textbf{Nash equilibrium} of the bi-objective control problem if
$$
\begin{cases}
\mathcal{J}_1(\mathbf{u}_1, \mathbf{u}_2) \leq \mathcal{J}_1(\mathbf{v}_1, \mathbf{u}_2), & \forall \mathbf{v}_1 \in U_1, \\
\mathcal{J}_2(\mathbf{u}_1, \mathbf{u}_2) \leq \mathcal{J}_2(\mathbf{u}_1, \mathbf{v}_2), & \forall \mathbf{v}_2 \in U_2.
\end{cases}
$$
\end{definition}
A triplet $(\mathbf{y}, \mathbf{u}_1, \mathbf{u}_2)$ is called a Nash solution of the bi-objective control problem if $(\mathbf{u}_1, \mathbf{u}_2)$ constitutes a Nash equilibrium of the corresponding reduced cost functionals $\mathcal{J}_1$ and $\mathcal{J}_2$. In other words, each player seeks to minimize their own cost functional while treating the other player's control as fixed. The equilibrium is reached when neither player can unilaterally improve their performance.

The optimization theory can be described in a \textit{player-wise} manner as follows. For a given control $\mathbf{u}_2$, \textbf{Player 1} seeks to minimize the cost functional~\ref{eq:cost_player1}, while for a given control $\mathbf{u}_1$, \textbf{Player 2} seeks to minimize the cost functional~\ref{eq:cost_player2}. This formulation gives rise to a \textit{non-cooperative bi-objective optimization problem}, in which each player aims to steer the common state toward their respective target while being influenced by the control actions of the other player.
\vspace{0.3cm}

Equivalently, a Nash equilibrium can also be expressed in the $argmin$ form: 
\begin{definition}[Nash Equilibrium]
A pair of controls $(\mathbf{u}_1, \mathbf{u}_2) \in \mathbf{U}_{1} \times \mathbf{U}_{2}$ is called a \emph{Nash equilibrium} if
\begin{align*}
\mathbf{u}_1 &= \argmin_{\mathbf{v}_1 \in \mathbf{U}_{1}} \, \mathcal{J}_1(\mathbf{y}(\mathbf{v}_1, \mathbf{u}_2), \mathbf{v}_1),\\
\mathbf{u}_2 &= \argmin_{\mathbf{v}_2 \in \mathbf{U}_{2}} \, \mathcal{J}_2(\mathbf{y}(\mathbf{u}_1, \mathbf{v}_2), \mathbf{v}_2),
\end{align*}
where $\mathbf{y}(\mathbf{v}_1, \mathbf{v}_2)$ denotes the state corresponding to the control pair $(\mathbf{v}_1, \mathbf{v}_2)$.  
\end{definition}

The \texttt{}{argmin} characterization of the Nash equilibrium provides a systematic structure for deriving these first-order optimality conditions, which are essential for analyzing and computing the equilibrium controls.

According to the definition of a Nash equilibrium, a pair of controls 
which is equivalent to stating that no player can decrease their cost by unilaterally changing their control. This definition naturally leads to the first-order optimality conditions derived from the Gâteaux derivative of the cost functionals. In particular, the Nash equilibrium controls satisfy
\begin{align}
\frac{\partial \mathcal{J}_1}{\partial \mathbf{u}_1}(\mathbf{u}_1, \mathbf{u}_2) &= 0, \label{eq:nash_eq1} \\
\frac{\partial \mathcal{J}_2}{\partial \mathbf{u}_2}(\mathbf{u}_1, \mathbf{u}_2) &= 0. \label{eq:nash_eq2}
\end{align}
where $\mathbf{u}_1 = \mathbf{u}_1(\mathbf{u}_2)$ and $\mathbf{u}_2 = \mathbf{u}_2(\mathbf{u}_1)$ represent the interdependence of the equilibrium controls. The existence and uniqueness of the Nash equilibrium can be established under standard convexity and continuity assumptions on the cost functionals and admissible sets.

\textbf{Remark:} The reduced cost functionals $\mathcal{J}_i(\mathbf{u}_1, \mathbf{u}_2)$ are convex in $\mathbf{u}_i$ for fixed $\mathbf{u}_j$ ($j \neq i$), ensuring that each player’s optimization problem has a unique minimizer. This convexity is crucial in proving the existence and uniqueness of the Nash equilibrium.

\subsection{Optimality Condition by Introducing the Adjoint Problem}

We consider a two-player unconstrained optimal control problem governed by the Stokes system. For $i=1,2$, the cost functional of player~$i$ is given by The reduced cost equation \eqref{eq:reduced_cost}

where the velocity--pressure pair $(\mathbf{y},\mathbf{p})$ satisfies the stationary Stokes equations \eqref{eq:stokes_state_system}.

Let $S : \mathbf{L}^2(\Omega)^d \to \mathbf{V}$ denote the \emph{Stokes solution operator} that maps the forcing term $\mathbf{g}$ to the corresponding divergence-free velocity $\mathbf{y} = S\mathbf{g}$.  
Then, we define the control-to-state operators
\begin{equation*}
S_i := S \circ B_i : \mathbf{U}_i \to \mathbf{H}.
\end{equation*}
With $\mathbf{y}_0 := S \mathbf{f}$ representing the homogeneous contribution due to $f$, the state variable can be expressed as
\begin{equation}\label{eq:state-representation}
\mathbf{y} = \mathbf{y}_0 + S_1 \mathbf{v}_1 + S_2 \mathbf{v}_2.
\end{equation}

Using \eqref{eq:state-representation}, the cost functional of player~$i$ becomes
\begin{equation*}
\mathcal{J}_i(\mathbf{v}_1,\mathbf{v}_2)
= \frac{1}{2}\|S_1 \mathbf{v}_1 + S_2 \mathbf{v}_2 + \mathbf{y}_0 - \mathbf{y}_{i,d}\|_{0,\Omega}^2
+ \frac{\alpha_i}{2}\|\mathbf{v}_i\|_{\mathbf{U}_i}^2.
\end{equation*}

The variational (first-order) optimality conditions for the Nash equilibrium of the unconstrained problem are given by \eqref{eq:nash_eq1}-\eqref{eq:nash_eq2}
\begin{align*}
0 &= \frac{\partial \mathcal{J}_1}{\partial \mathbf{u}_1}(\mathbf{u}_1, \mathbf{u}_2)[\mathbf{v}_1]
= \big( S_1 \mathbf{u}_1 + S_2 \mathbf{u}_2 + \mathbf{y}_0 - \mathbf{y}_{1,d}, S_1 \mathbf{v}_1 \big)
+ \alpha_1 \big(  \mathbf{u}_1, \mathbf{v}_1 \big),
\quad \forall \mathbf{v}_1 \in \mathbf{U}_1, \\[2pt]
0 &= \frac{\partial \mathcal{J}_1}{\partial \mathbf{u}_2}(\mathbf{u}_1, \mathbf{u}_2)[\mathbf{v}_2]
= \big(  S_1 \mathbf{u}_1 + S_2 \mathbf{u}_2 + \mathbf{y}_0 - \mathbf{y}_{2,d}, S_2 \mathbf{v}_2 \big)
+ \alpha_2 ( \mathbf{u}_2, \mathbf{v}_2 \big),
\quad \forall \mathbf{v}_2 \in \mathbf{U}_2. 
\end{align*}

 Now can be seen the Operator formulation using adjoints.
The above equalities are equivalent to the following 
\begin{equation}
\begin{aligned}
(S_1^{*} S_1 + \alpha_1 I)\mathbf{u}_1 + S_1^{*} S_2 \mathbf{u}_2 &= -S_1^{*}(\mathbf{y}_0 - \mathbf{y}_{1,d}), \\[4pt]
S_2^{*} S_1 \mathbf{u}_1 + (S_2^{*} S_2 + \alpha_2 I)\mathbf{u}_2 &= -S_2^{*}(\mathbf{y}_0 - \mathbf{y}_{2,d}).
\end{aligned}
\label{eq:coupled_optimality}
\end{equation}

where $S_1^*$ and $S_2^*$ denote the adjoint operators of $S_1$ and $S_2$, respectively.

Combining above equations, we obtain the linear block system(Coupled block system) 
\begin{equation}\label{eq:block-system}
\begin{pmatrix}S_1^* S_1 + \alpha_1 I & S_1^* S_2 \\[1ex]
S_2^* S_1 & S_2^* S_2 + \alpha_2 I
\end{pmatrix}
\begin{pmatrix}
\mathbf{u}_1 \\[1ex] \mathbf{u}_2
\end{pmatrix}
=
\begin{pmatrix}
- S_1^*(\mathbf{y}_0 - \mathbf{y}_{1,d}) \\[1ex] - S_2^*(\mathbf{y}_0 - \mathbf{y}_{2,d})
\end{pmatrix}.
\end{equation}

\paragraph{Remark.}  
The diagonal blocks represent each player’s self-interaction via the control-to-state mapping and regularization, while the off-diagonal blocks describe the coupling between players through the shared state variable $\mathbf{y}$.

\subsubsection{Existence and uniqueness}

Define the operator
\[
R : \mathbf{U}_1 \times \mathbf{U}_2 \to \mathbf{U}_1 \times \mathbf{U}_2, \qquad
R\begin{pmatrix}\mathbf{v}_1\\ \mathbf{v}_2\end{pmatrix}
=
\begin{pmatrix}
S_1^*(S_1 \mathbf{v}_1 + S_2 \mathbf{v}_2) + \alpha_1 \mathbf{v}_1 \\[4pt]
S_2^*(S_1 \mathbf{v}_1 + S_2 \mathbf{v}_2) + \alpha_2 \mathbf{v}_2
\end{pmatrix}.
\]
Then the system \eqref{eq:block-system} can be rewritten as
\[
R \mathbf{u} = \mathbf{z}, \qquad
\mathbf{u} = (\mathbf{u}_1, \mathbf{u}_2), \quad
\mathbf{z} = \big(-S_1^*(\mathbf{y}_0 - \mathbf{y}_{1,d}), -S_2^*(\mathbf{y}_0 - \mathbf{y}_{2,d})\big).
\]

We associate to $R$ the bilinear form $d : (\mathbf{U}_1 \times \mathbf{U}_2) \times (\mathbf{U}_1 \times \mathbf{U}_2) \to \mathbb{R}$ defined by
\[
d(\mathbf{v},\mathbf{w}) := (R \mathbf{v}, \mathbf{w})_{\mathbf{U}_1 \times \mathbf{U}_2}, \qquad
\mathbf{v}, \mathbf{w} \in \mathbf{U}_1 \times \mathbf{U}_2.
\]

Lets  Evaluate  the quadratic form. For any $\mathbf{v} =(\mathbf{v}_1,\mathbf{v}_2)$, we have
\begin{align*}
d(\mathbf{v},\mathbf{v})
&= (R \mathbf{v}, \mathbf{v})_{\mathbf{U}_1 \times \mathbf{U}_2} \\[2pt]
&= \big( S_1^*(S_1 \mathbf{v}_1 + S_2 \mathbf{v}_2), \mathbf{v}_1 \big)
+ \big( S_2^*(S_1 \mathbf{v}_1 + S_2 \mathbf{v}_2), \mathbf{v}_2 \big)
+ \alpha_1 \|\mathbf{v}_1\|_{\mathbf{U}_1}^2
+ \alpha_2 \|\mathbf{v}_2\|_{\mathbf{U}_2}^2 \\[2pt]
&= \|S_1 \mathbf{v}_1 + S_2 \mathbf{v}_2\|_{\mathbf{L}^2(\Omega)}^2
+ \alpha_1 \|\mathbf{v}_1\|_{\mathbf{U}_1}^2
+ \alpha_2 \|\mathbf{v}_2\|_{\mathbf{U}_2}^2.
\end{align*}

Since $\alpha_1, \alpha_2 > 0$ and $S_1, S_2$ are bounded, it follows that
\[
d(\mathbf{v},\mathbf{v})
\ge \alpha_1 \|\mathbf{v}_1\|_{\mathbf{U}_1}^2 + \alpha_2 \|\mathbf{v}_2\|_{\mathbf{U}_2}^2
\ge c \|\mathbf{v}\|_{\mathbf{U}_1 \times \mathbf{U}_2}^2,
\quad c := \min\{\alpha_1, \alpha_2\} > 0,
\]
showing that $d(\cdot,\cdot)$ is coercive  on $\mathbf{U}_1 \times \mathbf{U}_2$.

By the boundedness of $S_1$ and $S_2$, there exists $C>0$ such that
\[
|d(\mathbf{v},\mathbf{w})|
\le C\big(\|\mathbf{v}_1\|_{\mathbf{U}_1}+\|\mathbf{v}_2\|_{\mathbf{U}_2}\big)
     \big(\|\mathbf{w}_1\|_{\mathbf{U}_1}+\|\mathbf{w}_2\|_{\mathbf{U}_2}\big)
\le C\,\|\mathbf{v}\|_{\mathbf{U}_1\times\mathbf{U}_2}\,
        \|\mathbf{w}\|_{\mathbf{U}_1\times\mathbf{U}_2}.
\]

By the Lax--Milgram theorem, there exists a unique $\mathbf{u} = (\mathbf{u}_1, \mathbf{u}_2) \in \mathbf{U}_1 \times \mathbf{U}_2$ satisfying
\[
d(\mathbf{u},\mathbf{w}) = (\mathbf{z},\mathbf{w})_{\mathbf{U}_1 \times \mathbf{U}_2}, \qquad
\forall \mathbf{w} \in \mathbf{U}_1 \times \mathbf{U}_2.
\]
Equivalently, the block system \eqref{eq:block-system} admits a unique solution.

Once the \emph{optimal controls} $(\mathbf{u}_1,\mathbf{u}_2)$ are obtained from the coupled control system, the corresponding \emph{optimal state} $\mathbf{y}$ and \emph{adjoint variables} $(\boldsymbol{\varphi}_1,\boldsymbol{\varphi}_2)$ are uniquely determined. From \eqref{eq:state-representation} the optimal state variable satisfies
\begin{equation}\label{eq:optimal state-representation}
    \mathbf{y} = \mathbf{y}_0 + S_1 \mathbf{u}_1 + S_2 \mathbf{u}_2,
\end{equation}
where $S_i$ denotes the control-to-state operator associated with player~$i$. In particular, $S_i = S \circ B_i$, with $S$ being the Stokes solution operator and $B_i$ the control action operator.

\medskip

To identify the adjoint variables, recall that the Fr\'echet derivative of the control-to-state map at $(\mathbf{u}_1,\mathbf{u}_2)$ is the linear operator
\[
\mathbf{y}'(\mathbf{u}_1,\mathbf{u}_2): \mathbf{U}_1 \times \mathbf{U}_2 \to \mathbf{L}^2(\Omega),
\qquad 
\mathbf{y}'(\mathbf{u}_1,\mathbf{u}_2)[\delta \mathbf{u}_1, \delta \mathbf{u}_2]
= S_1 \delta \mathbf{u}_1 + S_2 \delta \mathbf{u}_2.
\]
Its adjoint operator is given by
\[
\big(\mathbf{y}'(\mathbf{u}_1,\mathbf{u}_2)\big)^{*}: \mathbf{L}^2(\Omega) \to \mathbf{U}_1 \times \mathbf{U}_2, 
\qquad 
\big(\mathbf{y}'(\mathbf{u}_1,\mathbf{u}_2)\big)^{*}\mathbf{w} = \big(S_1^{*}\mathbf{w},\, S_2^{*}\mathbf{w}\big),
\]
where $S_i^{*}$ denotes the $L^2$-adjoint of $S_i$.

For each player $i=1,2$, the adjoint variable is defined by
\[
\boldsymbol{\varphi}_i := S_i^{*}(\mathbf{y} - \mathbf{y}_{i,d}),
\]
which quantifies the sensitivity of the cost functional $\mathcal{J}_i$ with respect to the control $\mathbf{u}_i$ through the state dependence.  
Since $S_i = S \circ B_i$, this can equivalently be expressed as
\[
\boldsymbol{\varphi}_i = B_i^{*} S^{*}(\mathbf{y} - \mathbf{y}_{i,d}),
\]
showing explicitly that the adjoint operator maps the tracking residual in the state space back to the control space.

\medskip

In the PDE setting, $\boldsymbol{\varphi}_i$ is the unique solution of the adjoint Stokes problem whose forcing corresponds to the residual $\mathbf{y} - \mathbf{y}_{i,d}$, namely
\begin{equation}\label{adjoint}
-\nu \Delta \boldsymbol{\varphi}_i + \nabla q_i = \mathbf{y} - \mathbf{y}_{i,d},
\qquad 
\nabla \cdot \boldsymbol{\varphi}_i = 0,
\qquad 
\boldsymbol{\varphi}_i|_{\partial \Omega} = 0.
\end{equation}
Under the standard regularity assumptions on the Stokes system, this problem admits a unique solution $(\boldsymbol{\varphi}_i, q_i) \in \mathbf{V} \times Q$.

\medskip

The first-order (stationarity) condition for player~$i$ is obtained by setting the derivative of $\mathcal{J}_i$ with respect to $\mathbf{u}_i$ to zero, yielding
\[
S_i^{*}(\mathbf{y} - \mathbf{y}_{i,d}) + \alpha_i \mathbf{u}_i = 0,
\qquad \text{or equivalently,} \qquad
S_i^{*}\mathbf{y} + \alpha_i \mathbf{u}_i = S_i^{*}\mathbf{y}_{i,d}.
\]
Substituting the state representation 
$\mathbf{y} = \mathbf{y}_0 + S_1 \mathbf{u}_1 + S_2 \mathbf{u}_2$
leads to the coupled optimality system for the controls \eqref{eq:coupled_optimality}.

In summary, for each player $i=1,2$, the adjoint variable is defined by 
\[
\boldsymbol{\varphi}_i = B_i^{*} S^{*}(\mathbf{y} - \mathbf{y}_{i,d}),
\]
which satisfies the adjoint Stokes system \eqref{adjoint}.
The corresponding first-order \emph{optimality condition} is
\[
B_i^{*}\boldsymbol{\varphi}_i + \alpha_i \mathbf{u}_i = 0,
\]
and, together with the \emph{optimal state solution} \eqref{eq:optimal state-representation}, 
these relations constitute the coupled optimality system for the pair of controls $(\mathbf{u}_1,\mathbf{u}_2)$.

This establishes a direct connection between the abstract operator formulation \(S_i^*(\mathbf{y}-\mathbf{y}_{i,d})\) and its concrete PDE representation, providing a rigorous characterization of the Nash equilibrium for the two-player Stokes control problem.

\noindent\textbf{Remarks on well-posedness.}
We work under the standing assumptions that $\mathbf{U}_1,\mathbf{U}_2$ are Hilbert
spaces, the control-to-state operators $S_i:\mathbf{U}_i\to\mathbf{L}^2(\Omega)$
are bounded linear maps, and $\alpha_i>0$ for $i=1,2$. With these hypotheses the
block operator
\[
R(\mathbf{v}_1,\mathbf{v}_2)^\top :=
\begin{pmatrix}
S_1^*(S_1 \mathbf{v}_1 + S_2 \mathbf{v}_2) + \alpha_1 \mathbf{v}_1\\[2pt]
S_2^*(S_1 \mathbf{v}_1 + S_2 \mathbf{v}_2) + \alpha_2 \mathbf{v}_2
\end{pmatrix}
\]
is associated to the continuous bilinear form
$d(\mathbf{v},\mathbf{w})=(R\mathbf{v},\mathbf{w})_{\mathbf{U}_1\times\mathbf{U}_2}$.
Since
$d(\mathbf{v},\mathbf{v})=\|S_1\mathbf{v}_1+S_2\mathbf{v}_2\|_{0,\Omega}^2
+\alpha_1\|\mathbf{v}_1\|_{\mathbf{U}_1}^2+\alpha_2\|\mathbf{v}_2\|_{\mathbf{U}_2}^2$,
coercivity and continuity follow and the Lax--Milgram theorem yields a unique
solution $\mathbf{u}=(\mathbf{u}_1,\mathbf{u}_2)\in\mathbf{U}_1\times\mathbf{U}_2$
of $R\mathbf{u}=\mathbf{z}$. The corresponding state $\mathbf{y}=\mathbf{y}_0+S_1\mathbf{u}_1+S_2\mathbf{u}_2$
and the adjoints $\boldsymbol{\varphi}_i=S_i^*(\mathbf{y}-\mathbf{y}_{i,d})$
are then uniquely determined by the well-posed state and adjoint Stokes problems.

Combining the \emph{state}, \emph{adjoint}, and \emph{optimality condition} relations, we obtain the following Theorem.  
This system provides the necessary and sufficient conditions for the unconstrained Nash equilibrium, ensuring the existence and uniqueness of the optimal pair $(\mathbf{u}_1,\mathbf{u}_2)$ together with their associated state and adjoint variables.

\begin{theorem}
The tuple $(\mathbf{y}, p, \mathbf{u}_1, \mathbf{u}_2) \in \mathbf{H} \times Q \times \mathbf{U}_1 \times \mathbf{U}_2$ is a Nash equilibrium of problem \eqref{optimal control problem-1} if and only if there exist $(\boldsymbol{\varphi}_i, r_i) \in \mathbf{H} \times Q$ (for $i=1,2$), such that
$$
(\mathbf{y}, p, \mathbf{u}_1, \mathbf{u}_2, \boldsymbol{\varphi}_1, r_1, \boldsymbol{\varphi}_2, r_2) \in \mathbf{H} \times Q \times \mathbf{U}_1 \times \mathbf{U}_2 \times \mathbf{H} \times Q \times \mathbf{H} \times Q
$$
satisfies, for $i=1,2$,
\begin{equation}\label{eq:optimality-system-nash}
\begin{cases}
a(\mathbf{y}, \mathbf{w}) + b(\mathbf{w}, p) = (\mathbf{f} + B_1 \mathbf{u}_1 + B_2 \mathbf{u}_2,\, \mathbf{w}), & \forall\, \mathbf{w} \in \mathbf{H}, \\[4pt]
b(\mathbf{y}, q) = 0, & \forall\, q \in Q, \\[4pt]
a(\boldsymbol{\varphi}_i, \mathbf{w}) + b(\mathbf{w}, r_i) = (\mathbf{y} - \mathbf{y}_{i,d},\, \mathbf{w}), & \forall\, \mathbf{w} \in \mathbf{H}, \\[4pt]
b(\boldsymbol{\varphi}_i, q) = 0, & \forall\, q \in Q, \\[4pt]
B_i^{*}\boldsymbol{\varphi}_i + \alpha_i \mathbf{u}_i = 0, & \text{in } \Omega.
\end{cases}
\end{equation}
\end{theorem}
The system \eqref{eq:optimality-system-nash} thus couples, for each player $i=1,2$, the state, adjoint, and optimality relations, and completely characterizes the Nash equilibrium of the bi-objective optimal control problem.

\section{ Control Problem and Finite Element Approximation}\label{sec:3}

We develop a finite element approximation for the proposed problem and derive the optimality conditions for both the continuous (exact) and discrete formulations. Furthermore, we establish \textit{a priori} error estimates that quantify the difference between the exact solution and its finite element approximation in the $L^2$- and $H^1$-norms. A related work by Rahman and Borz{\'i} \cite{rahman2015fem} presented a finite element–multigrid scheme for elliptic Nash equilibrium multiobjective optimal control problems, which serves as a reference point for the present study.

\subsection{Finite element approximation}

For simplicity, the computational domain $\Omega$ is assumed to be a polygon when $\Omega \subset \mathbb{R}^2$, or a polyhedron when $\Omega \subset \mathbb{R}^3$. The domain $\Omega$ is partitioned into a family of quasi-uniform triangulations $\mathscr{T}^h = \{T\}$, where each element $T$ is a triangle if $d=2$, or a tetrahedron if $d=3$ with maximum element diameter
$$
h := \max_{T \in \mathscr{T}^h} \{\operatorname{diam}(T)\}.
$$
Similarly, let $\mathscr{T}_{U{i}}^h = \{T_{U_{i}}\}$ be a family of quasi-uniform triangulations for the control space, characterized by the maximum mesh size
$$
h_{U_i} := \max_{T_{U_{i}} \in \mathscr{T}_{U_{i}}^h} \{\operatorname{diam}(T_{U_i})\}.
$$

Associated with the triangulation $\mathscr{T}^{h}$, there are the finite element subspaces $\mathbf{H}^{h} \subset \mathbf{H}$ and $Q^{h} \subset Q$, chosen such that the pair $\left(\mathbf{H}^{h}, Q^{h}\right)$ satisfies the discrete inf–sup (LBB) condition. That is, there exists a constant $b_{l}^{\prime} > 0$, independent of $h$, such that
\begin{equation}\label{dicrete llb}
\inf _{q_{h} \in Q^{h}} \sup _{\mathbf{z}_{h} \in \mathbf{H}^{h}} \frac{b\left(\mathbf{z}_{h}, q_{h}\right)}{\left\|\mathbf{z}_{h}\right\|_{1 ; \Omega}\left\|q_{h}\right\|_{0 ; \Omega}} \geq b_{l}^{\prime} . 
\end{equation}
Moreover, there exist two integers $m \geq 1$ and $n \geq 1$ such that
$$
\begin{aligned}
\inf _{z_{h} \in \mathbf{H}^{h}}\left\|\mathbf{z}-\mathbf{z}_{h}\right\|_{0 ; \Omega}+h\left\|\mathbf{z}-\mathbf{z}_{h}\right\|_{1 ; \Omega} & \leq C h^{m+1}\|\mathbf{z}\|_{m+1 ; \Omega}, \quad \forall \mathbf{z} \in \mathbf{H} \cap \mathbf{H}^{m+1}(\Omega), \\
\inf _{q_{h} \in Q^{h}}\left\|q-q_{h}\right\|_{0 ; \Omega} & \leq C h^{n+1}\|q\|_{n+1 ; \Omega}, \quad ~~ \forall q \in Q \cap H^{n+1}(\Omega).
\end{aligned}
$$

The above assumptions are satisfied by the Taylor-Hood finite elements ($\mathbf{P}_{m}, P_{n}$) when $m=n+1$, and by the Mini elements $\left(\mathbf{P}_{m} \oplus \text{Bubble}, P_{n}\right)$ with $m=n=1$, see \cite{brenner2008mathematical, ciarlet2002finite, girault1996}. Corresponding to the triangulation $\mathscr{T}_{U_i}^{h}$ we define another finite-dimensional subspace:
\begin{equation*}
\mathbf{U}_i^{h} := \Bigl\{ \mathbf{v}_{i,h} \in \mathbf{U}_i : \left. \mathbf{v}_{i,h} \right|_{T_{U_i}} \text{ are polynomials of degree } \le k \, (0 \le k \le m), \, \forall T_{U_i} \in \mathscr{T}_{U_i}^{h} \Bigr\}
\end{equation*}
such that
\begin{equation}   \label{error}
\inf _{\mathbf{v}_{i,h} \in \mathrm{U}^{h}}\left\|\mathbf{v}_i-\mathbf{v}_{i,h}\right\|_{0 ; \Omega} \leq C h_{U_i}^{k+1}\|\mathbf{v}_i\|_{k+1 ; \Omega}, \quad \forall \mathbf{v}_i \in \mathbf{U}_i \cap \mathbf{H}^{k+1}(\Omega). 
\end{equation}
Hence, the finite element formulation of the problem \eqref{optimal control problem-1} - \eqref{optimal_control_proble-2} reads:
\begin{equation}  \label{fem optimal control}
\min _{\mathbf{y}_{h} \in \mathbf{K}^{h}} \mathcal{J}_i\left(\mathbf{y}_{h}, \mathbf{u}_{i,h}\right)
\end{equation}
subject to
$$
\begin{cases}a\left(\mathbf{y}_{h}, \mathbf{w}_{h}\right)+b\left(\mathbf{w}_{h}, p_{h}\right)=\left(\mathbf{f}+B_1 \mathbf{u}_{1,h} + B_2 \mathbf{u}_{2,h}, \mathbf{w}_{h}\right) & \forall \mathbf{w}_{h} \in \mathbf{H}^{h}, \\ b\left(\mathbf{y}_{h}, q_{h}\right)=0 & \forall q_{h} \in Q^{h}.\end{cases}
$$
Similarly, we obtain the optimality conditions of problem \eqref{fem optimal control}, which is stated in the following theorem.

\begin{theorem}
    
The tuple $\left(\mathbf{y}_{h}, p_{h}, \mathbf{u}_{1,h}, \mathbf{u}_{1,h}\right) \in \mathbf{H}^{h} \times Q^{h} \times \mathbf{U}_1^{h} \times \mathbf{U}_2^{h}$ is the solution to the problem \ref{fem optimal control} if and only if there exist  tuples $\left(\boldsymbol{\varphi}_{i, h}, r_{i,h} \right) \in \mathbf{H}^{h} \times Q^{h} $ such that $\left(\mathbf{y}_{h}, p_{h}, \mathbf{u}_{i,h}, \boldsymbol{\varphi}_{i, h}, r_{i,h}\right) \in$ $\mathbf{H}^{h} \times Q^{h} \times \mathbf{U}_i^{h} \times \mathbf{H}^{h} \times Q^{h}$ satisfies the following optimality conditions:

\begin{equation}
     \quad \begin{cases}a\left(\mathbf{y}_{h}, \mathbf{w}_{h}\right)+b\left(\mathbf{w}_{h}, p_{h}\right)=\left(\mathbf{f}+B_1 \mathbf{u}_{1,h}+ B_2\mathbf{u}_{2,h}, \mathbf{w}_{h}\right) & \forall \mathbf{w}_{h} \in \mathbf{H}^{h}, \\ b\left(\mathbf{y}_{h}, q_{h}\right)=0 & \forall q_{h} \in Q^{h},  \\ a\left(\boldsymbol{\varphi}_{i, h}, \mathbf{w}_{h}\right)+b\left(\mathbf{w}_{h}, r_{i,h}\right)=\left(\mathbf{y}_{h}-\mathbf{y}_{i,d}, \mathbf{w}_{h}\right) & \forall \mathbf{w}_{h} \in \mathbf{H}^{h}, \\ b\left(\boldsymbol{\varphi}_{i, h}, q_{h}\right)=0 & \forall q_{h} \in Q^{h}, \\ \alpha_i \mathbf{u}_{i,h}+\mathcal{P}_{U_i}^{h} B_i^{*}\left(\boldsymbol{\varphi}_{i, h}\right)=0, & \text { in } \Omega,\end{cases}\label{discretized optimality condition}
\end{equation}
where  satisfy $\mathcal{P}_{U_i}^{h}$ is the $L^{2}$-projection operator from $\mathbf{U}_i$ to $\mathbf{U}^{h}_i$ such that

$$
\left(\mathcal{P}_{U_i}^{h} \mathbf{v}_i, \mathbf{v}_{i,h}\right)=\left(\mathbf{v}_i, \mathbf{v}_{i,h}\right), \quad \forall \mathbf{v}_i \in \mathbf{U}_i, \mathbf{v}_{i,h} \in \mathbf{U}^{h}_i
.$$
\end{theorem} 
It is clear that $\mathcal{P}_{U_i}^{h}$ defines a linear operator, as $\mathbf{U}_i^{h}$ constitutes a linear subspace of the Banach space $\mathbf{U}_i$. Moreover, the first-order optimality conditions \eqref{discretized optimality condition} are sufficient because the state equations are linear and the cost functional $\mathcal{J}_i(\cdot)$ is convex. Hence, the discrete problem \eqref{fem optimal control} is equivalent to the system \eqref{discretized optimality condition}.

\subsection{\textit{A Priori} Estimates}
In this section, we study the convergence behavior of the algorithm. While our study is inspired by the work of \cite{niu2011stokes}, it extends the analysis by formulating and deriving \textit{a priori} error estimates for a Nash equilibrium problem governed by the Stokes equations. This extension represents the main theoretical novelty of the paper, as it rigorously addresses the interplay between multiple control objectives within a non-cooperative optimal control problem. For this analysis, we consider the operator $B_{i}$ is reversible from $\mathbf{L}^{2}(\Omega)$ to itself and from $\mathbf{H}^{1}(\Omega)$ to itself. 

\begin{proposition}
For the discrete solution $(\mathbf{y}_{h}, p_{h})$ of the problem \eqref{discretized optimality condition}, the following stability estimate holds, independently of $h$ and $h_{U}$, such that
\begin{equation}   \label{eq:stability_of _discrete_solution}
\left\|\mathbf{y}_{h}\right\|_{1 ; \Omega}+ \left\|p_{h}\right\|_{0 ; \Omega}+\sum_{i=1}^2 \Big(\left\|\mathbf{u}_{i,h}\right\|_{0 ; \Omega}+\left\|\boldsymbol{\varphi}_{i, h}\right\|_{1 ; \Omega}+\left\|r_{i,h}\right\|_{0 ; \Omega}\Big) \leq C
\end{equation}
\end{proposition}
\begin{proof}
To establish the result, we take $\mathbf{w}_{h} = \mathbf{y}_{h}$ and $q_{h} = p_{h}$ in equation \eqref{discretized optimality condition}.
\end{proof}

Hence, there exists a subsequence that converges weakly to a solution of problem \eqref{eq:optimality-system-nash} as $h \to 0$. Since the solution of problem \eqref{eq:optimality-system-nash} is unique, it follows that the entire sequence $(\mathbf{u}_{i,h}, \mathbf{y}_{h}, p_{h}, \boldsymbol{\varphi}_{i,h}, r_{i,h})$ converges weakly to the exact solution $(\mathbf{u}, \mathbf{y}, p, \boldsymbol{\varphi}_i, r_i)$. We now proceed to establish \textit{a priori} error estimates between the exact and finite element solutions. In this study, we employ two types of finite element spaces: the Taylor-Hood element $(\mathbf{P}_{l+1}, P_{l}; l \ge 1)$ and the Mini-element $(\mathbf{P}_{l} \oplus \text{Bubble}, P_{l}; l = 1)$. Under these choices, we assume that the solution of the optimality system possesses the regularity properties described in \cite{cattabriga1961problema};
\begin{equation}   \label{eq:3.1}
\mathbf{y}, \boldsymbol{\varphi}_i \in \mathbf{H}^{l+2}(\Omega), \quad p, r_i \in H^{l+1}(\Omega) . 
\end{equation}
The following two theorems state the $H^{1}$-norm and $L^{2}$-norm error estimates, respectively. The constant $l$ is given in \eqref{eq:3.1}, and $k \le l + 1$ is as defined in \eqref{error}.

\begin{theorem}   \label{therem-1}
Let $(\mathbf{y}, p, \boldsymbol{\varphi}_i, r_i, \mathbf{u}_i)$ and $(\mathbf{y}_h, p_h, \boldsymbol{\varphi}_{i,h}, r_{i,h}, \mathbf{u}_{i,h})$ be the solutions to problems \eqref{eq:optimality-system-nash} and \eqref{discretized optimality condition}, respectively. The $H^{1} \times L^{2}$-norm error bounds for the velocity and pressure variables are stated as follows:
\begin{equation*}
\begin{aligned}
& \left\|\mathbf{y}-\mathbf{y}_{h}\right\|_{1, \Omega} + \left\|\boldsymbol{\varphi}_1 - \boldsymbol{\varphi}_{1,h}\right\|_{1, \Omega} + \left\|\boldsymbol{\varphi}_2 - \boldsymbol{\varphi}_{2,h}\right\|_{1, \Omega} \\
& + \left\|p - p_{h}\right\|_{0, \Omega} + \left\|r_1 - r_{1,h}\right\|_{0, \Omega} + \left\|r_2 - r_{2,h}\right\|_{0, \Omega} 
& \leq C\left(h^{l+1} + h_{U_1}^{k+2} + h_{U_2}^{k+2}\right)
\end{aligned}
\end{equation*}
and $L^{2}$-norm error estimate for the control as follows:
\begin{equation*}
\left\|\mathbf{u}_i-\mathbf{u}_{i,h}\right\|_{0 ; \Omega} \leq C\left(h^{l+2}+h_{U_i}^{k+1}\right) 
\end{equation*}
\end{theorem}
\medskip

\begin{theorem}\label{theorem-2}
Let $\left(\mathbf{y}, p, \boldsymbol{\varphi}_i, r_i, \mathbf{u}_i\right)$ be the solutions of \eqref{eq:optimality-system-nash} and $\left(\mathbf{y}_{h}, p_{h}, \boldsymbol{\varphi}_{i,h}, r_{i,h}, \mathbf{u}_{i,h}\right)$  be the solutions of \eqref{discretized optimality condition}, then the following $L^{2}$-norm error estimates holds:
\begin{equation*}
\begin{split}
&  \left\|\mathbf{y} - \mathbf{y}_{h}\right\|_{0, \Omega} + \left\|\boldsymbol{\varphi}_1 - \boldsymbol{\varphi}_{1,h}\right\|_{0, \Omega} + \left\|\boldsymbol{\varphi}_2 - \boldsymbol{\varphi}_{2,h}\right\|_{0, \Omega}  \\
& + \left\|\mathcal{P}_{U_1}^{h} \mathbf{u}_1 - \mathbf{u}_{1,h}\right\|_{0, \Omega} + \left\|\mathcal{P}_{U_2}^{h} \mathbf{u}_2 - \mathbf{u}_{2,h}\right\|_{0, \Omega} \leq C \left(h^{l+2} + h_{U_1}^{k+2} + h_{U_2}^{k+2}\right)
\end{split}
\end{equation*}

\end{theorem} 
The proofs of Theorems \ref{therem-1} and \ref{theorem-2} are established based on the following five lemmas. To this end, we first introduce the corresponding auxiliary equations:
\begin{equation}
\begin{cases}
a\left(\mathbf{y}_{h}(\mathbf{u}_1,\mathbf{u}_2), \mathbf{w}_{h}\right) + b\left(\mathbf{w}_{h}, p_{h}(\mathbf{u}_1,\mathbf{u}_2)\right) = \left(\mathbf{f} + B_1 \mathbf{u}_1 + B_2\mathbf{u}_2, \mathbf{w}_{h}\right), & \forall \mathbf{w}_{h} \in \mathbf{H}^{h},  \\
b\left(\mathbf{y}_{h}(\mathbf{u}_1,\mathbf{u}_2), q_{h}\right) = 0, & \forall q_{h} \in Q^{h}, \\
a\left(\boldsymbol{\varphi}_{i, h}(\mathbf{u}_1,\mathbf{u}_2), \mathbf{w}_{h}\right) + b\left(\mathbf{w}_{h}, r_{i,h}(\mathbf{u}_1,\mathbf{u}_2)\right) = \left( \mathbf{y} - \mathbf{y}_{i,d}, \mathbf{w}_{h}\right), & \forall \mathbf{w}_{h} \in \mathbf{H}^{h}, \\ 
b\left(\boldsymbol{\varphi}_{i, h}(\mathbf{u}_1,\mathbf{u}_2), q_{h}\right) = 0, & \forall q_{h} \in Q^{h}.
\end{cases} \label{auxiliary optimality condition}
\end{equation}

\subsubsection{Proof of the Main Result}
Let $\mathbf{u}_i$ and $\mathbf{u}_{i,h}$ denote the continuous and discrete control variables solving \eqref{eq:optimality-system-nash} and \eqref{discretized optimality condition}, respectively, for $i = 1, 2$. Let $\mathcal{P}_{U_i}^{h}: \mathbf{L}^2(\Omega) \to \mathbf{U}_i^{h}$ be the $L^2$-projection operator onto the discrete control space of player $i$. We denote by $\mathbf{y}(\mathbf{u}_1, \mathbf{u}_2)$ the continuous state corresponding to the controls $(\mathbf{u}_1, \mathbf{u}_2)$, and by $\mathbf{y}_h(\mathbf{u}_1, \mathbf{u}_2)$ the associated discrete state. Similarly, $\boldsymbol{\varphi}_i$ and $\boldsymbol{\varphi}_{i,h}(\mathbf{u}_1, \mathbf{u}_2)$ represent the continuous and discrete adjoint variables for player $i$, respectively.

\noindent We begin by estimating
$$
\left\|\mathbf{y}_{h}(\mathbf{u}_1,\mathbf{u}_2) - \mathbf{y}_{h}\right\|_{1 ; \Omega} + \left\|p_{h}(\mathbf{u}_1,\mathbf{u}_2) - p_{h}\right\|_{0 ; \Omega},
$$ 
followed by the estimation of
$$
\left\|\boldsymbol{\varphi}_{i, h}(\mathbf{u}_1,\mathbf{u}_2) - \boldsymbol{\varphi}_{i, h}\right\|_{1 ; \Omega} + \left\| r_{i,h}(\mathbf{u}_1,\mathbf{u}_2) - r_{i,h} \right\|_{0 ; \Omega}.
$$
\begin{lemma}  \label{lemma1}
Let $(\mathbf{y}_{h}(\mathbf{u}_1,\mathbf{u}_2), p_{h}(\mathbf{u}_1,\mathbf{u}_2))$ and $(\mathbf{y}_{h}, p_{h})$ be the solutions of auxiliary equation ~\eqref{auxiliary optimality condition} and discretized equation  \eqref{discretized optimality condition}, respectively. Then, the following estimates hold:
\begin{equation}
\begin{split}   \label{eq:lemma1_result}
& \left\|\mathbf{y}_{h}(\mathbf{u}_1, \mathbf{u}_2) - \mathbf{y}_{h}\right\|_{1;\Omega} + \left\|p_{h}(\mathbf{u}_1, \mathbf{u}_2) - p_{h}\right\|_{0;\Omega}  \\
\leq ~ & C \Big( \left\|\mathcal{P}_{U_1}^{h} \mathbf{u}_1 - \mathbf{u}_{1,h}\right\|_{0;\Omega} + h_{U_1} \left\|\mathcal{P}_{U_1}^{h} \mathbf{u}_1 - \mathbf{u}_1\right\|_{0;\Omega} \notag \\
& + \left\|\mathcal{P}_{U_2}^{h} \mathbf{u}_2 - \mathbf{u}_{2,h}\right\|_{0;\Omega} + h_{U_2} \left\|\mathcal{P}_{U_2}^{h} \mathbf{u}_2 - \mathbf{u}_2\right\|_{0;\Omega} \Big)
\end{split}
\end{equation}
\end{lemma} 

\begin{proof}
From \eqref{discretized optimality condition} and \eqref{auxiliary optimality condition}, it follows that
\begin{equation}
\begin{split}
& a\bigl(\mathbf{y}_{h}(\mathbf{u}_1, \mathbf{u}_2)-\mathbf{y}_{h}, \mathbf{w}_{h}\bigr) + b\bigl(\mathbf{w}_{h}, p_{h}(\mathbf{u}_1, \mathbf{u}_2)-p_{h}\bigr) \\
& \qquad \qquad \qquad \qquad \qquad = \bigl(B_1(\mathbf{u}_1-\mathbf{u}_{1,h}), \mathbf{w}_{h}\bigr) + \bigl(B_2(\mathbf{u}_2-\mathbf{u}_{2,h}), \mathbf{w}_{h}\bigr), \quad \forall \mathbf{w}_{h} \in \mathbf{H}^{h}, \\
& b\bigl(\mathbf{y}_{h}(\mathbf{u}_1, \mathbf{u}_2)-\mathbf{y}_{h}, q_{h}\bigr) = 0, \quad \forall q_{h} \in Q^{h}.
\end{split}   \label{eq:difference_lemma1}
\end{equation}
Using the definition of the projection operator
$$
\begin{aligned}
& \sum_{i=1,2}\mid\left(\mathbf{u}_i-\mathbf{u}_{i,h}, B_i^{*} \mathbf{w}_{h}\right)\mid \\
\leq ~ & \sum_{i=1,2} \left(\left(\mathbf{u}_i-\mathcal{P}_{U_i}^{h} \mathbf{u}_i, B_i^{*} \mathbf{w}_{h}-\mathcal{P}_{U_i}^{h} B_i^{*} \mathbf{w}_{h}\right) + \sum_{i=1,2} \left|\left(\mathcal{P}_{U_i}^{h} \mathbf{u}_i-\mathbf{u}_{i,h}, B_i^{*} \mathbf{w}_{h}\right)\right|\right).
\end{aligned}
$$
Now we bound this by applying Cauchy–Schwarz and Poincar{\'e} inequalities along with the estimate in the equation \eqref{error}
$$
\begin{aligned}
\sum_{i=1,2}\mid\left(\mathbf{u}_i-\mathbf{u}_{i,h}, B_i^{*} \mathbf{w}_{h}\right)\mid 
& \leq C \sum_{i=1,2}\left(\left\|\mathcal{P}_{U_i}^{h} \mathbf{u}_i-\mathbf{u}_{i,h}\right\|_{0 ; \Omega}+h_{U_i}\left\|\mathbf{u}_i-\mathcal{P}_{U_i}^{h} \mathbf{u}_i\right\|_{0 ; \Omega}\right)\left\|B_i^{*} \mathbf{w}_{h}\right\|_{1 ; \Omega} \\
& \leq C \sum_{i=1,2} \left( \left\|\mathcal{P}_{U_i}^{h} \mathbf{u}_i-\mathbf{u}_{i,h}\right\|_{0 ; \Omega}+h_{U_i}\left\|\mathbf{u}_i-\mathcal{P}_{U_i}^{h} \mathbf{u}_i\right\|_{0 ; \Omega}\right)\left\|\mathbf{w}_{h}\right\|_{1 ; \Omega}
\end{aligned}
$$
Next, setting $\mathbf{w}_{h} = \mathbf{y}_{h}(\mathbf{u}_1, \mathbf{u}_2) - \mathbf{y}_{h}$, we apply coercivity \eqref{inequality} in \eqref{eq:difference_lemma1}, we obtain
$$
\left\|\mathbf{y}_{h}(\mathbf{u_1,u_2}) - \mathbf{y}_{h}\right\|_{1 ; \Omega} \leq C \left(\sum_{i=1,2}\left\|\mathcal{P}_{U_i}^{h} \mathbf{u}_i-\mathbf{u}_{i,h}\right\|_{0 ; \Omega} + h_{U_i}\left\|\mathcal{P}_{U_i}^{h} \mathbf{u}_i - \mathbf{u}_i\right\|_{0 ; \Omega}\right).
$$
Further, using LBB-condition \eqref{dicrete llb}, we can obtain
$$
\begin{aligned}
b_{l}^{\prime}\left\|p_{h}(\mathbf{u_1,u_2})-p_{h}\right\|_{0 ; \Omega} & \leq \sup _{\mathbf{w}_{h} \in \mathbf{H}^{h}} \frac{1}{\left\|\mathbf{w}_{h}\right\|_{1 ; \Omega}} \left(\sum_{i=1,2}\left|\left(\mathbf{u}_i-\mathbf{u}_{i,h}, B_i^{*} \mathbf{w}_{h}\right) \right| + \left|a\left(\mathbf{y}_{h}(\mathbf{u_1,u_2}) - \mathbf{y}_{h}, \mathbf{w}_{h}\right)\right|\right) \\
& \leq C \sum_{i=1,2}\left(\left\|\mathcal{P}_{U_i}^{h} \mathbf{u}_i - \mathbf{u}_{i,h}\right\|_{0 ; \Omega} + h_{U_i}\left\|\mathcal{P}_{U_i}^{h} \mathbf{u}_i - \mathbf{u}_i\right\|_{0 ; \Omega}\right).
\end{aligned}
$$
By combining the two inequalities above, we obtain \eqref{eq:lemma1_result}, which concludes the proof of the lemma.
\end{proof}

Further, we estimate the terms
$$
\left\|\boldsymbol{\varphi}_{i, h}(\mathbf{u_1,u_2})-\mathbf{y}_{h}\right\|_{1 ; \Omega} ~~~ \text{and} ~~~ \left\|r_{i,h}(\mathbf{u_1,u_2})-p_{h}\right\|_{0 ; \Omega}.
$$

\begin{lemma} \label{lemma2}
Let $\left(\boldsymbol{\varphi}_{i, h}(\mathbf{u}_1,\mathbf{u}_2), r_{i,h}(\mathbf{u}_1,\mathbf{u}_2)\right)$ and $\left(\boldsymbol{\varphi}_{i, h}, r_{i,h}\right)$ be the solutions of the adjoint state of auxiliary equation \eqref{auxiliary optimality condition} and discretized equation \eqref{discretized optimality condition}, respectively. Let $\mathbf{y}_{h}(\mathbf{u}_1,\mathbf{u}_2)$ and $\mathbf{y}_{h}$ be the corresponding auxiliary and  discrete state solutions. Then there holds the following inequality:
\begin{equation}\label{result2}
\left\|\boldsymbol{\varphi}_{i,h}(\mathbf{u}_1,\mathbf{u}_2) - \boldsymbol{\varphi}_{i,h}\right\|_{1;\Omega}
+ \left\|r_{i,h}(\mathbf{u}_1,\mathbf{u}_2) - r_{i,h}\right\|_{0;\Omega}
\leq C \left\|\mathbf{y} - \mathbf{y}_{h}\right\|_{0;\Omega}.
\end{equation}
\end{lemma}

\begin{proof}
Subtracting the third and fourth equations of \eqref{discretized optimality condition} from those of \eqref{auxiliary optimality condition} yields the following error system :
$$
\begin{aligned}
a\left(\boldsymbol{\varphi}_{i, h}(\mathbf{u}_1,\mathbf{u}_2) - \boldsymbol{\varphi}_{i, h}, \mathbf{w}_{h}\right) + b\left(\mathbf{w}_{h}, r_{i,h}(\mathbf{u}_1,\mathbf{u}_2) - r_{i,h}\right) & = \left(\mathbf{y}-\mathbf{y}_{h}, \mathbf{w}_{h}\right), \quad \forall \mathbf{w}_{h} \in \mathbf{H}^{h},  \\
b\left(\boldsymbol{\varphi}_{i, h}(\mathbf{u}_1,\mathbf{u}_2)-\boldsymbol{\varphi}_{i, h}, q_{h}\right) & = 0, \qquad \qquad \qquad \qquad \qquad \forall q_{h} \in Q^{h}
\end{aligned}
$$
By choosing
$$
\mathbf{w}_{h} = \boldsymbol{\varphi}_{i, h}(\mathbf{u}_1,\mathbf{u}_2) - \boldsymbol{\varphi}_{i, h} \in \mathbf{H}^{h}
$$
the term $b(\mathbf{w}_{h}, q_{h})$ in the first equation vanishes, since $\boldsymbol{\varphi}_{i,h}(\mathbf{u}_1,\mathbf{u}_2) - \boldsymbol{\varphi}_{i,h}$ satisfies the divergence-free condition
$$
b(\boldsymbol{\varphi}_{i,h}(\mathbf{u}_1,\mathbf{u}_2) - \boldsymbol{\varphi}_{i,h}, q_{h}) = 0.
$$
This leads to:
$$
a\left(\boldsymbol{\varphi}_{i, h}(\mathbf{u}_1,\mathbf{u}_2)-\boldsymbol{\varphi}_{i, h}, \boldsymbol{\varphi}_{i, h}(\mathbf{u}_1,\mathbf{u}_2)-\boldsymbol{\varphi}_{i, h}\right)= \left(\mathbf{y}-\mathbf{y}_{h}, \boldsymbol{\varphi}_{i, h}(\mathbf{u}_1,\mathbf{u}_2)-\boldsymbol{\varphi}_{i, h}\right).
$$
Applying the $\mathbf{H}^1$-coercivity of the bilinear form $a(\cdot, \cdot)$, the Cauchy–Schwarz inequality, and the continuous embedding $\|\cdot\|_{0;\Omega} \leq C \|\cdot\|_{1;\Omega}$, we obtain:
$$
\left\|\boldsymbol{\varphi}_{i, h}(\mathbf{u}_1,\mathbf{u}_2)-\boldsymbol{\varphi}_{i, h}\right\|_{1 ; \Omega} \leq C \left\|\mathbf{y}-\mathbf{y}_{h}\right\|_{0 ; \Omega}.
$$

\noindent Applying the discrete inf-sup (LBB) condition \eqref{dicrete llb} to $r_{i,h}(\mathbf{u}_1,\mathbf{u}_2) - r_{i,h}$, we have:
\begin{align*}
\beta \left\|r_{i,h}(\mathbf{u}_1,\mathbf{u}_2) - r_{i,h}\right\|_{0 ; \Omega} & \leq \sup _{\mathbf{w}_{h} \in \mathbf{H}^{h}} \frac{b\left(\mathbf{w}_{h}, r_{i,h}(\mathbf{u}_1,\mathbf{u}_2) - r_{i,h}\right)}{\left\|\mathbf{w}_{h}\right\|_{1 ; \Omega}}\\
& = \sup _{\mathbf{w}_{h} \in \mathbf{H}^{h}} \frac{1}{\left\|\mathbf{w}_{h}\right\|_{1 ; \Omega}} \left(\left|\left(\mathbf{y} - \mathbf{y}_{h}, \mathbf{w}_{h}\right)\right| + \left|a\left(\boldsymbol{\varphi}_{i, h}(\mathbf{u}_1,\mathbf{u}_2) - \boldsymbol{\varphi}_{i, h}, \mathbf{w}_{h}\right)\right|\right).
\end{align*}
By using the Cauchy–Schwarz inequality, the continuity property of $a(\cdot, \cdot)$, and the $\mathbf{H}^1$-estimate for $\tilde{\boldsymbol{\varphi}}_{i,h}$, the above expression reduces to:
\begin{align*}
\left\| r_{i,h}(\mathbf{u}_1, \mathbf{u}_2) - r_{i,h} \right\|_{0;\Omega} & \le C \Bigl(\left\| \mathbf{y} - \mathbf{y}_{h} \right\|_{0;\Omega} + \left\| \boldsymbol{\varphi}_{i,h}(\mathbf{u}_1, \mathbf{u}_2) - \boldsymbol{\varphi}_{i,h} \right\|_{1;\Omega} \Bigr) \\
& \le C \left\| \mathbf{y} - \mathbf{y}_{h} \right\|_{0;\Omega}.
\end{align*}
Adding the estimates for $\boldsymbol{\varphi}_{i, h}(\mathbf{u}_1,\mathbf{u}_2) -\boldsymbol{\varphi}_{i, h}$ and $r_{i,h}(\mathbf{u}_1,\mathbf{u}_2) - r_{i,h}$ completes the proof.
\end{proof}

\begin{remark}
\begin{equation}\label{adjoint-bound}
\left\|\boldsymbol{\varphi}_{i,h}(\mathbf{u}_1,\mathbf{u}_2) - \boldsymbol{\varphi}_{i,h} \right\|_{0;\Omega} \leq \left\|\boldsymbol{\varphi}_{i,h}(\mathbf{u}_1,\mathbf{u}_2) - \boldsymbol{\varphi}_{i,h} \right\|_{1;\Omega}\leq \|\mathbf{y}-\mathbf{y}_h\|_{0,\Omega}
\end{equation}
\end{remark}

\begin{lemma}[Off-diagonal cross-term bound]\label{lemma3}
Let $\phi_{j,h}:=\boldsymbol{\varphi}_{j,h}(\mathbf{u}_1,\mathbf{u}_2)
-\boldsymbol{\varphi}_{j,h}$,
then for any $\varepsilon>0$ there exists a constant $C>0$ independent of $h$
such that
\begin{align*}
&\Big|\sum_{\substack{i,j=1\\ i\ne j}}^{2}
(B_i(\mathbf{u}_i-\mathbf{u}_{i,h}),\,\phi_{j,h})\Big|\\
\le & ~ C \varepsilon\Big(
\sum_i\|\mathcal P_{U_i}^h\mathbf{u}_i-\mathbf{u}_{i,h}\|_{0,\Omega}^2
+ \sum_i\|\mathbf{u}_i-\mathcal P_{U_i}^h\mathbf{u}_i\|_{0,\Omega}^2
+ \|\mathbf{y}_h(\mathbf{u}_{1},\mathbf{u}_{2})-\mathbf{y}_h\|_{0,\Omega}^2
+ \sum_j\|\nabla\phi_{j,h}\|_{0,\Omega}^2
\Big)
\notag\\
&\quad + C\varepsilon^{-1}\Big( \|\mathbf{y}_h(\mathbf{u}_{1},\mathbf{u}_{2})-\mathbf{y}_h\|_{0,\Omega}^2 + \sum_i\|\mathbf{u}_i-\mathcal P_{U_i}^h\mathbf{u}_i\|_{0,\Omega}^2 \Big).
\end{align*}
\end{lemma}

\begin{proof}
Define
\[
\phi_{j,h} := \boldsymbol{\varphi}_{j,h}(\mathbf{u}_1,\mathbf{u}_2)
-\boldsymbol{\varphi}_{j,h}.
\]

We want to bound 
\[
S:=\sum_{\substack{i,j=1\\ i\ne j}}^{2} (B_i(\mathbf{u}_i-\mathbf{u}_{i,h}),\phi_{j,h}).
\]

First, split the control difference. For each $i$,
\[
\mathbf{u}_i-\mathbf{u}_{i,h}
= (\mathcal P_{U_i}^h\mathbf{u}_i-\mathbf{u}_{i,h})
+ (\mathbf{u}_i-\mathcal P_{U_i}^h\mathbf{u}_i).
\]
Define
\[
S^{(d)} := \sum_{i\ne j} (B_i(\mathcal P_{U_i}^h\mathbf{u}_i-\mathbf{u}_{i,h}),\phi_{j,h}),
\qquad
S^{(p)} := \sum_{i\ne j} (B_i(\mathbf{u}_i-\mathcal P_{U_i}^h\mathbf{u}_i),\phi_{j,h}).
\]
Then $S=S^{(d)}+S^{(p)}$.

Next, using the Cauchy--Schwarz and Poincar{\'e} inequality 
\[
|S^{(d)}|
\le C_B \Big(\sum_i \|\mathcal P_{U_i}^h\mathbf{u}_i-\mathbf{u}_{i,h}\|_0\Big)
\Big(\sum_j\|\phi_{j,h}\|_0\Big),
\]
\[
|S^{(p)}|
\le C_B \Big(\sum_i \|\mathbf{u}_i-\mathcal P_{U_i}^h\mathbf{u}_i\|_0\Big)
\Big(\sum_j\|\phi_{j,h}\|_0\Big).
\]
Thus, using $\|\phi_{j,h}\|_0\le C_P\|\phi_{j,h}\|_1$,
\begin{equation}
\label{eq:S_preYoung}
|S| \le C\,\Big(\sum_i\|\mathcal P_{U_i}^h\mathbf{u}_i-\mathbf{u}_{i,h}\|_0 + \sum_i\|\mathbf{u}_i-\mathcal P_{U_i}^h\mathbf{u}_i\|_0\Big) \Big(\sum_j\|\phi_{j,h}\|_1\Big).
\end{equation}

Further, using Young's inequality with a small parameter $\varepsilon>0$. We \emph{enlarge} the $\varepsilon$-group to include $\|\mathbf{y}_h(\mathbf{u}_{1},\mathbf{u}_{2})-\mathbf{y}_h\|_{0, \Omega}^2$ and
$\sum_j\|\nabla\phi_{j,h}\|_{0, \Omega}^2$, since these will later be absorbed into the coercive left-hand side of the global energy inequality. Hence
\begin{align}
|S| & \le C\varepsilon\Big(\sum_i\|\mathcal P_{U_i}^h\mathbf{u}_i-\mathbf{u}_{i,h}\|_{0, \Omega}^2 + \sum_i\|\mathbf{u}_i-\mathcal P_{U_i}^h\mathbf{u}_i\|_{0, \Omega}^2 + \|\mathbf{y}_h(\mathbf{u}_{1},\mathbf{u}_{2})-\mathbf{y}_h\|_{0, \Omega}^2 + \sum_j\|\nabla\phi_{j,h}\|_{0, \Omega}^2\Big) \notag\\
&\quad + C\varepsilon^{-1}\Big(\|\mathbf{y}_h(\mathbf{u}_{1},\mathbf{u}_{2})-\mathbf{y}_h\|_{0, \Omega}^2 + \sum_j\|\phi_{j,h}\|_1^2 + \sum_i\|\mathbf{u}_i-\mathcal P_{U_i}^h\mathbf{u}_i\|_{0, \Omega}^2\Big). \label{eq:S_afterYoung}
\end{align}

Now our aim is to find the bound for the term $\sum_j\|\phi_{j,h}\|_1^2$. Subtracting the discrete adjoint equations for $\boldsymbol{\varphi}_{j,h}(\mathbf{u}_{1}, \mathbf{u}_{1})$ and $\boldsymbol{\varphi}_{j,h}$ yields
\begin{equation}
a(w_h,\phi_{j,h}) + b(w_h,\pi_{j,h}) = (\mathbf{y}_h(\mathbf{u}_{1}, \mathbf{u}_{2}) - \mathbf{y}_h,w_h),
\end{equation}
where $\pi_{j,h}=p_{j,h}^*(\mathbf{u}_{1}, \mathbf{u}_{2})-p_{j,h}^*$.

Choose $w_h=\phi_{j,h}$ in \eqref{eq:S_preYoung} and using coercivity and Poicar{\'e} inequality, we get
\[
c_a\|\phi_{j,h}\|_1^2
\le C_P\|\mathbf{y}_h(\mathbf{u}_{1},\mathbf{u}_{2})-\mathbf{y}_h\|_0\|\phi_{j,h}\|_1.
\]
Thus
\begin{equation}
\|\phi_{j,h}\|_1
\le C
\|\mathbf{y}_h(\mathbf{u}_{1},\mathbf{u}_{2})-\mathbf{y}_h\|_0.
\end{equation}
Squaring,
\begin{equation}  \label{eq:phijh_bound}
\|\phi_{j,h}\|_1^2 \le C \|\mathbf{y}_h(\mathbf{u}_{1},\mathbf{u}_{2})-\mathbf{y}_h\|_{0, \Omega}^2.
\end{equation}
Summing over $j=1,2$ in the equation \eqref{eq:phijh_bound},
\[
\sum_j\|\phi_{j,h}\|_1^2
\le C\Big(
\|\mathbf{y}_h(\mathbf{u}_{1},\mathbf{u}_{2})-\mathbf{y}_h\|_{0, \Omega}^2
+ \|\mathbf{y}-\mathbf{y}_h(\mathbf{u}_{1},\mathbf{u}_{2})\|_{0, \Omega}^2\,\Theta^2
\Big).
\]
Insert into \eqref{eq:S_afterYoung}:
\begin{align}
|S|
&\le C\varepsilon\Big(
\sum_i\|\mathcal P_{U_i}^h\mathbf{u}_i-\mathbf{u}_{i,h}\|_{0, \Omega}^2
+ \sum_i\|\mathbf{u}_i-\mathcal P_{U_i}^h\mathbf{u}_i\|_{0, \Omega}^2
+ \|\mathbf{y}_h(\mathbf{u}_{1},\mathbf{u}_{2})-\mathbf{y}_h\|_{0, \Omega}^2
+ \sum_j\|\nabla\phi_{j,h}\|_{0, \Omega}^2
\Big)
\notag\\
&\quad + C\varepsilon^{-1}\Big(
\|\mathbf{y}_h(\mathbf{u}_{1},\mathbf{u}_{2})-\mathbf{y}_h\|_{0, \Omega}^2
+ \sum_i\|\mathbf{u}_i-\mathcal P_{U_i}^h\mathbf{u}_i\|_{0, \Omega}^2
\Big).
\end{align}

\end{proof}

Next, we derive a priori error estimates for the discrete Nash equilibrium controls, states, and adjoints. 
The analysis proceeds by first establishing stability and consistency results for the discrete state and adjoint equations, 
which quantify the approximation quality of the finite-dimensional subspaces. 
These preliminary results are then combined to control the difference between the continuous and discrete optimality systems. 
A key technical difficulty arises from the coupling of the two control variables through the cost functionals, 
leading to cross-player terms in the discrete variational identities. 
To handle these, we introduce a dedicated auxiliary estimate (Lemma~\ref{lemma2}), 
which bounds the mixed contributions in terms of adjoint and projection errors. 
This bound is then instrumental in the proof of the main control-error estimate (Lemma~\ref{lemma4}), 
where all small $\varepsilon$--weighted terms are absorbed into the coercive left-hand side, 
yielding the final error estimate~\eqref{eq:nash-control-error}. 
The resulting bounds ensure optimal convergence rates of the discrete controls under standard regularity assumptions.

Further, we estimate the following term $\sum_{i=1}^2\left\|\mathcal{P}^{h} \mathbf{u}_i - \mathbf{u}_{i,h}\right\|_{0;\Omega}$.

\begin{lemma}[Control-error estimate]    \label{lemma4}
Let $(\mathbf{u}_1,\mathbf{u}_2)$ and $(\mathbf{u}_{1,h},\mathbf{u}_{2,h})$ as the continuous and discrete Nash equilibrium controls corresponding to Eqs.~\eqref{eq:optimality-system-nash} and \eqref{discretized optimality condition}, respectively. Consider the $L^2$-projection $\mathcal{P}_{U_i}^{h}: L^2(\Omega)^m \to U_{h,i}$ onto the discrete control space for player $i=1,2$. Denote by $\mathbf{y}(\mathbf{u}_1,\mathbf{u}_2)$ and $\mathbf{y}_h (\mathbf{u}_1, \mathbf{u}_2)$ the continuous and discrete state solutions, and by $\boldsymbol{\varphi}_i$, and $\boldsymbol{\varphi}_{i,h}(\mathbf{u}_1,\mathbf{u}_2)$ the corresponding continuous and discrete adjoints for player $i$. Then, there exists a constant $C>0$, independent of the mesh size $h$, such that
\begin{equation}\label{eq:nash-control-error}
\begin{aligned}
& \left\|\mathcal{P}_{U_1}^{h}\mathbf{u}_1-\mathbf{u}_{1,h}\right\|_{0,\Omega} + \left\|\mathcal{P}_{U_2}^{h}\mathbf{u}_2-\mathbf{u}_{2,h}\right\|_{0,\Omega}  \\
\le ~ & C \Big(\|\mathbf{y} - \mathbf{y}_h(\mathbf{u}_1,\mathbf{u}_2)\|_{0,\Omega} + \|\boldsymbol{\varphi}_1 - \boldsymbol{\varphi}_{1,h}(\mathbf{u}_1,\mathbf{u}_2)\|_{0,\Omega} + \|\boldsymbol{\varphi}_2-\boldsymbol{\varphi}_{2,h}(\mathbf{u}_1,\mathbf{u}_2)\|_{0,\Omega} \\
& + h_{U_1}\|\mathbf{u}_1-\mathcal{P}_{U_1}^{h}\mathbf{u}_1\|_{0,\Omega} + h_{U_2}\|\mathbf{u}_2-\mathcal{P}_{U_2}^{h}\mathbf{u}_2\|_{0,\Omega}\Big).
\end{aligned}
\end{equation}
\end{lemma}
\begin{proof}
From the continuous optimality for player $i$,
\[
B_i^*\boldsymbol{\varphi}_i + \alpha_i \mathbf{u}_i = 0 \quad\text{in }\Omega,
\]
apply the $L^2$--projection $\mathcal P_{U_i}^h$ to obtain
\[
\mathcal P_{U_i}^h B_i^*\boldsymbol{\varphi}_i + \alpha_i \mathcal P_{U_i}^h\mathbf{u}_i = 0.
\]
Subtract the discrete optimality relation
\[
\alpha_i \mathbf{u}_{i,h} + \mathcal P_{U_i}^h B_i^*\boldsymbol{\varphi}_{i,h} = 0
\]
to get, for each $i=1,2$,
\begin{equation}\label{pf:opt-diff-1}
\alpha_i\big(\mathcal P_{U_i}^h\mathbf{u}_i - \mathbf{u}_{i,h}\big)
+ \mathcal P_{U_i}^h B_i^*\big(\boldsymbol{\varphi}_i - \boldsymbol{\varphi}_{i,h}\big) = 0.
\end{equation}

Take the $L^2$--inner product of \eqref{pf:opt-diff-1} with $\mathcal P_{U_i}^h\mathbf{u}_i - \mathbf{u}_{i,h}$ and sum over $i=1,2$. Using the projection orthogonality
\((\mathcal P_{U_i}^h v, w_h) = (v, w_h)\) for all $w_h\in U_{h,i}$, we obtain
\begin{equation}   \label{eq:auxiliary_control_error}
\sum_{i=1}^2 \alpha_i \|\mathcal P_{U_i}^h\mathbf{u}_i - \mathbf{u}_{i,h}\|_{0,\Omega}^2
= -\sum_{i=1}^2 \big(B_i(\mathcal P_{U_i}^h\mathbf{u}_i - \mathbf{u}_{i,h}),\ \boldsymbol{\varphi}_i - \boldsymbol{\varphi}_{i,h}\big).
\end{equation}

From the continuous and discrete optimality conditions for each player $i=1,2$, we have
$$
\alpha_i \big(\mathcal{P}_{U_i}^{h}\mathbf{u}_i - \mathbf{u}_{i,h}, \mathcal{P}_{U_i}^{h}\mathbf{u}_i - \mathbf{u}_{i,h} \big)=\alpha_i \big(\mathcal{P}_{U_i}^{h}\mathbf{u}_i - \mathbf{u}_{i,h}, \mathbf{u}_i - \mathbf{u}_{i,h} \big) = - \big(\mathcal{P}_{U_i}^{h}\mathbf{u}_i - \mathbf{u}_{i,h},\, B_i^*\boldsymbol{\varphi}_i - \mathcal{P}_{U_i}^{h}B_i^*\boldsymbol{\varphi}_{i,h} \big).
$$

Now split each adjoint difference as
\[
\boldsymbol{\varphi}_i - \boldsymbol{\varphi}_{i,h}
= \big(\boldsymbol{\varphi}_i - \boldsymbol{\varphi}_{i,h}(\mathbf{u}_1,\mathbf{u}_2)\big)
+ \big(\boldsymbol{\varphi}_{i,h}(\mathbf{u}_1,\mathbf{u}_2) - \boldsymbol{\varphi}_{i,h}\big).
\]

Expanding the differences using the projection $\mathcal{P}_{U_i}^{h}$, we obtain
$$
\begin{aligned}
&\sum_{i=1}^{2} \alpha_i \|\mathcal{P}_{U_i}^{h}\mathbf{u}_i - \mathbf{u}_{i,h}\|_{0,\Omega}^2 \\
= & \sum_{i=1}^{2} \alpha_i\big(\mathcal{P}_{U_i}^{h}\mathbf{u}_i-\mathbf{u}_{i,h},\, \mathcal{P}_{U_i}^{h}\mathbf{u}_i-\mathbf{u}_{i,h}\big) =\sum_{i=1}^{2} \big(\mathcal{P}_{U_i}^{h}\mathbf{u}_i-\mathbf{u}_{i,h},\, \alpha_i\mathcal{P}_{U_i}^{h}\mathbf{u}_i-\alpha_i\mathbf{u}_{i,h}\big) \\
= & -\sum_{i=1}^{2} \big(\mathcal{P}_{U_i}^{h}\mathbf{u}_i - \mathbf{u}_{i,h},\, B_i^*\boldsymbol{\varphi}_i - B_i^*\boldsymbol{\varphi}_{i,h} \big) =  -\sum_{i=1}^{2} \big(B_i(\mathcal{P}_{U_i}^{h}\mathbf{u}_i - \mathbf{u}_{i,h}),\, \boldsymbol{\varphi}_i - \boldsymbol{\varphi}_{i,h} \big) \\
= & -\sum_{i=1}^{2} \big(B_i(\mathcal{P}_{U_i}^{h}\mathbf{u}_i - \mathbf{u}_{i,h}),\, \boldsymbol{\varphi}_i - \boldsymbol{\varphi}_{i,h}(\mathbf{u}_1,\mathbf{u}_2) \big)  \;\; -\sum_{i=1}^{2}\big(B_i(\mathcal{P}_{U_i}^{h}\mathbf{u}_i - \mathbf{u}_{i,h}),\, \boldsymbol{\varphi}_{i,h}(\mathbf{u}_1,\mathbf{u}_2) - \boldsymbol{\varphi}_{i,h} \big)\\
= & -\sum_{i=1}^{2} \big(B_i(\mathcal{P}_{U_i}^{h}\mathbf{u}_i - \mathbf{u}_{i,h}),\, \boldsymbol{\varphi}_i - \boldsymbol{\varphi}_{i,h}(\mathbf{u}_1,\mathbf{u}_2) \big) \;\; -\sum_{i=1}^{2}\big(B_i(\mathbf{u}_i - \mathbf{u}_{i,h}),\, \boldsymbol{\varphi}_{i,h}(\mathbf{u}_1,\mathbf{u}_2) - \boldsymbol{\varphi}_{i,h} \big)\\
& ~~~~~~~~~~~~~~ -\sum_{i=1}^{2}\big(B_i(\mathcal{P}_{U_i}^{h}\mathbf{u}_i - \mathbf{u}_i), \boldsymbol{\varphi}_{i,h}(\mathbf{u}_1,\mathbf{u}_2) - \boldsymbol{\varphi}_{i,h} \big)\\
= & -\sum_{i=1}^{2}  \big(B_i(\mathcal{P}_{U_i}^{h}\mathbf{u}_i - \mathbf{u}_{i,h}),\, \boldsymbol{\varphi}_i - \boldsymbol{\varphi}_{i,h}(\mathbf{u}_1,\mathbf{u}_2)\big) \;\; - \sum_{\substack{i,j=1\\ i\ne j}}^{2} \big(B_j(\mathbf{u}_j - \mathbf{u}_{j,h}),\, \boldsymbol{\varphi}_{i,h}(\mathbf{u}_1,\mathbf{u}_2) - \boldsymbol{\varphi}_{i,h}\big)  \\
& \qquad \qquad \;\; -\sum_{i=1}^{2}a\big(\mathbf{y}_h(\mathbf{u}_1,\mathbf{u}_2) - \mathbf{y}_h,\, \boldsymbol{\varphi}_i(\mathbf{u}_1,\mathbf{u}_2) - \boldsymbol{\varphi}_{i,h}\big)  \;\; -\sum_{i=1}^{2} \big(B_i(\mathcal{P}_{U_i}^{h}\mathbf{u}_i - \mathbf{u}_i),\, \boldsymbol{\varphi}_{i,h}(\mathbf{u}_1,\mathbf{u}_2) - \boldsymbol{\varphi}_{i,h}\big) \\
= & -\sum_{i=1}^{2} \big(B_i(\mathcal{P}_{U_i}^{h}\mathbf{u}_i - \mathbf{u}_{i,h}),\, \boldsymbol{\varphi}_i - \boldsymbol{\varphi}_{i}(\mathbf{u}_1,\mathbf{u}_2)\big)  - \sum_{\substack{i,j=1\\ i\ne j}}^{2}\big(B_i(\mathbf{u}_i - \mathbf{u}_{i,h}),\, \boldsymbol{\varphi}_{j,h}(\mathbf{u}_1,\mathbf{u}_2) - \boldsymbol{\varphi}_{j,h}\big) \\
&\qquad \qquad\;\;   - \sum_{i=1}^{2} \big(\mathbf{y} - \mathbf{y}_h,\, \mathbf{y}_h(\mathbf{u}_1,\mathbf{u}_2) -  \mathbf{y}_h\big) - \sum_{i=1}^{2}\big(B_i(\mathcal{P}_{U_i}^{h}\mathbf{u}_i -\mathbf{u}_i),\, \boldsymbol{\varphi}_{i,h}(\mathbf{u}_1,\mathbf{u}_2) - \boldsymbol{\varphi}_{i,h}\big)\\
= & \underbrace{-\sum_{i=1}^{2}  \big(B_i(\mathcal{P}_{U_i}^{h}\mathbf{u}_i - \mathbf{u}_{i,h}),\, \boldsymbol{\varphi}_i - \boldsymbol{\varphi}_{i,h}(\mathbf{u}_1,\mathbf{u}_2)\big)}_{\textbf{I}}  - \underbrace{\sum_{\substack{i,j=1\\ i\ne j}}^{2} \big(B_i(\mathbf{u}_i - \mathbf{u}_{i,h}),\, \boldsymbol{\varphi}_{j,h}(\mathbf{u}_1,\mathbf{u}_2) - \boldsymbol{\varphi}_{j,h}\big)}_{\textbf{II}} \\
&\qquad \qquad\;\; - \underbrace{\sum_{i=1}^{2}\big(\mathbf{y} - \mathbf{y}_h ,\,\mathbf{y}_h(\mathbf{u}_1,\mathbf{u}_2) - \mathbf{y}_h\big)}_{\textbf{III}} 
\;\;- \underbrace{ \sum_{i=1}^{2}\big(B_i(\mathcal{P}_{U_i}^{h}\mathbf{u}_i - \mathbf{u}_i),\, \boldsymbol{\varphi}_{i,h}(\mathbf{u}_1,\mathbf{u}_2) - \boldsymbol{\varphi}_{i,h}\big)}_{\textbf{IV}}\\
\end{aligned}
$$

\medskip\noindent\textbf{(I) First sum (self-term):}
Use the boundedness of $B_i$ and Young's inequality. For any $\varepsilon>0$,
\[
\begin{aligned}
& \left|\left(B_i(\mathcal P_{U_i}^h\mathbf{u}_i - \mathbf{u}_{i,h}),\ \boldsymbol{\varphi}_i - \boldsymbol{\varphi}_{i,h}(\mathbf{u}_1,\mathbf{u}_2)\right)\right| \\
& \qquad \le \|B_i\|\;\|\mathcal P_{U_i}^h\mathbf{u}_i - \mathbf{u}_{i,h}\|_{0,\Omega}\;
\|\boldsymbol{\varphi}_i - \boldsymbol{\varphi}_{i,h}(\mathbf{u}_1,\mathbf{u}_2)\|_{0,\Omega}\\
& \qquad \le \varepsilon \|\mathcal P_{U_i}^h\mathbf{u}_i - \mathbf{u}_{i,h}\|_{0,\Omega}^2
+ \frac{\|B_i\|^2}{4\varepsilon}\,\|\boldsymbol{\varphi}_i - \boldsymbol{\varphi}_{i,h}(\mathbf{u}_1,\mathbf{u}_2)\|_{0,\Omega}^2.
\end{aligned}
\]
Summing over $i=1,2$ yields the corresponding bound for the first sum.

\medskip\noindent\textbf{(II) Second sum (cross-terms):} Using the Cauchy--Schwarz and Poincar{\'e} inequality and Young's inequality
\begin{align*}
&\Big|\sum_{\substack{i,j=1\\ i\ne j}}^{2}
(B_i(\mathbf{u}_i-\mathbf{u}_{i,h}),\,\phi_{j,h})\Big|\\
\le & ~ C \varepsilon\Big(
\sum_i\|\mathcal P_{U_i}^h\mathbf{u}_i-\mathbf{u}_{i,h}\|_{0,\Omega}^2
+ \sum_i\|\mathbf{u}_i-\mathcal P_{U_i}^h\mathbf{u}_i\|_{0,\Omega}^2
+ \|\mathbf{y}_h(\mathbf{u}_{1},\mathbf{u}_{2})-\mathbf{y}_h\|_{0,\Omega}^2
+ \sum_j\|\nabla\phi_{j,h}\|_{0,\Omega}^2
\Big)
\notag\\
&\quad + C\varepsilon^{-1}\Big( \|\mathbf{y}_h(\mathbf{u}_{1},\mathbf{u}_{2})-\mathbf{y}_h\|_{0,\Omega}^2 + \sum_i\|\mathbf{u}_i-\mathcal P_{U_i}^h\mathbf{u}_i\|_{0,\Omega}^2 \Big).
\end{align*}
From here we concluded that 
\[
\boxed{
\begin{aligned}
&\sum_{i\neq j}
\big|\big(B_j(\mathbf{u}_j-\mathbf{u}_{j,h}),\,
\boldsymbol{\varphi}_{i,h}(\mathbf{u}_1,\mathbf{u}_2)
-\boldsymbol{\varphi}_{i,h}\big)\big| \\
&\quad\le C\Big( \|\mathbf{y}_h(\mathbf{u}_{1},\mathbf{u}_{2})-\mathbf{y}_h\|_{0,\Omega}^2 + \sum_i\|\mathbf{u}_i-\mathcal P_{U_i}^h\mathbf{u}_i\|_{0,\Omega}^2 \Big).
\end{aligned}}
\]

\medskip\noindent\textbf{(III) Third sum:}
By the Cauchy--Schwarz and Young inequalities, we have
\[
-\big(\mathbf{y} - \mathbf{y}_h ,\,\mathbf{y}_h(\mathbf{u}_1,\mathbf{u}_2) - \mathbf{y}_h\big)
= \big(\mathbf{y}-\mathbf{y}_h(\mathbf{u}_1,\mathbf{u}_2)\,
\,,\mathbf{y}_h(\mathbf{u}_1,\mathbf{u}_2)-\mathbf{y}_h\big)-\|\mathbf{y}_h(\mathbf{u}_1,\mathbf{u}_2)-\mathbf{y}_h\|_{0,\Omega}^2
\]

where,
\[\big(\mathbf{y} - \mathbf{y}_h ,\,\mathbf{y}_h(\mathbf{u}_1,\mathbf{u}_2) - \mathbf{y}_h\big) \le \frac{1}{4\varepsilon}\|\mathbf{y}-\mathbf{y}_h(\mathbf{u}_1,\mathbf{u}_2)\|_{0,\Omega}+ \varepsilon
\|\mathbf{y}_h(\mathbf{u}_1,\mathbf{u}_2)-\mathbf{y}_h\|_{0,\Omega}
\]

For the diagonal terms (where the indices match) and the projection-error term
$\big(B_i(\mathcal P_{U_i}^h\mathbf{u}_i - \mathbf{u}_i),\,\cdot\big)$.

\medskip\noindent\textbf{(IV) Fourth sum:} 
We similarly obtain bound for 
\[
\big|\big(B_i(\mathcal P_{U_i}^h\mathbf{u}_i - \mathbf{u}_i),\,
\boldsymbol{\varphi}_{i,h}(\mathbf{u}_1,\mathbf{u}_2)
- \boldsymbol{\varphi}_{i,h}\big)\big|
\]
for small $\varepsilon$ (and constant $C$ depending on $\|B_i\|$).

By Cauchy--Schwarz and Young inequalities, it follows that
\[
\begin{aligned}
\sum_{i=1}^2
\bigl(
B_i(\mathcal{P}_{U_i}^{h}\mathbf{u}_i - \mathbf{u}_i),\,
\boldsymbol{\varphi}_{i,h}(\mathbf{u}_1,\mathbf{u}_2)
- \boldsymbol{\varphi}_{i,h}
\bigr)
&\le
\varepsilon \sum_{i=1}^2
\|\nabla(\boldsymbol{\varphi}_{i,h}(\mathbf{u}_1,\mathbf{u}_2)
- \boldsymbol{\varphi}_{i,h})\|_{0,\Omega}^2 \\
&\quad
+ \frac{1}{4\varepsilon} \sum_{i=1}^2
h^2 \|B_i\|^2
\|\mathcal{P}_{U_i}^{h}\mathbf{u}_i - \mathbf{u}_i\|_{0,\Omega}^2.
\end{aligned}
\]

\medskip\noindent\textbf{Collecting estimates and absorption.}
Combine the bounds from (I)--(IV).

\[
\begin{aligned}
&\sum_{i=1}^{2} \alpha_i \|\mathcal{P}_{U_i}^{h}\mathbf{u}_i - \mathbf{u}_{i,h}\|_{0,\Omega}^2 
+ \sum_{i=1}^{2} \|\mathbf{y}_h(\mathbf{u}_1,\mathbf{u}_2) - \mathbf{y}_h\|_{0,\Omega}^2 \\
&\leq \epsilon \sum_{i=1}^{2} \Big( \alpha_i \|\mathcal{P}_{U_i}^{h}\mathbf{u}_i - \mathbf{u}_{i,h}\|_{0,\Omega}^2
+ \|\mathbf{y}_h(\mathbf{u}_1,\mathbf{u}_2) - \mathbf{y}_h\|_{0,\Omega}^2
+ \| \boldsymbol{\varphi}_{i,h}(\mathbf{u}_1,\mathbf{u}_2) - \boldsymbol{\varphi}_{i,h} \|_{0,\Omega}^2 \Big) \\
&\quad + C \epsilon^{-1} \sum_{i=1}^{2} \Big( \|\mathbf{y} - \mathbf{y}_{h}(\mathbf{u}_1,\mathbf{u}_2)\|_{0,\Omega}^2
+ \|\boldsymbol{\varphi}_i - \boldsymbol{\varphi}_{i,h}(\mathbf{u}_1,\mathbf{u}_2)\|_{0,\Omega}^2
+ h_{U_i}^2 \|\mathbf{u}_i - \mathcal{P}_{U_i}^{h}\mathbf{u}_i \|_{0,\Omega}^2 \Big).
\end{aligned}
\]

Finally, choosing $\epsilon$ sufficiently small, we absorb the first group of terms into the left-hand side to obtain the discrete control error estimate. 

 Using the approximation property
of \(\mathcal P_{U_i}^h\) (or keeping the \(h_{U_i}\)-factor explicitly)
we obtain the stated estimate \eqref{eq:nash-control-error}, Finally using Lemma \ref{lemma1}-\ref{lemma4} we are concluded that :
\[
\begin{aligned}
& \|\mathcal P_{U_1}^h\mathbf{u}_1-\mathbf{u}_{1,h}\|_{0,\Omega}
+ \|\mathcal P_{U_2}^h\mathbf{u}_2-\mathbf{u}_{2,h}\|_{0,\Omega} \\
&\quad\le C\Big(\|\mathbf{y}(\mathbf{u}_1,\mathbf{u}_2)-\mathbf{y}_h(\mathbf{u}_1,\mathbf{u}_2)\|_{0,\Omega}
+ \|\boldsymbol{\varphi}_1-\boldsymbol{\varphi}_{1,h}(\mathbf{u}_1,\mathbf{u}_2)\|_{0,\Omega} \\
&\qquad\qquad + \|\boldsymbol{\varphi}_2-\boldsymbol{\varphi}_{2,h}(\mathbf{u}_1,\mathbf{u}_2)\|_{0,\Omega}
+ h_{U_1}\|\mathbf{u}_1-\mathcal P_{U_1}^h\mathbf{u}_1\|_{0,\Omega}
+ h_{U_2}\|\mathbf{u}_2-\mathcal P_{U_2}^h\mathbf{u}_2\|_{0,\Omega}\Big),
\end{aligned}
\]
with \(C\) independent of \(h\). 

Finally, choosing $\epsilon$ sufficiently small, we absorb the first group of terms into the left-hand side to obtain the discrete control error estimate
$$
\begin{aligned}
& \|\mathcal{P}_{U_1}^{h}\mathbf{u}_1 - \mathbf{u}_{1,h}\|_{0,\Omega} + \|\mathcal{P}_{U_2}^{h}\mathbf{u}_2 - \mathbf{u}_{2,h}\|_{0,\Omega}  \\
\le ~ &  C  \|\mathbf{y} - \mathbf{y}_{h}(\mathbf{u}_1,\mathbf{u}_2)\|_{0,\Omega} + C \sum_{i=1}^{2}\Big(\|\boldsymbol{\varphi}_i - \boldsymbol{\varphi}_{i,h}(\mathbf{u}_1,\mathbf{u}_2)\|_{0,\Omega} + h_{U_i} \|\mathbf{u}_i - \mathcal{P}_{U_i}^{h}\mathbf{u}_i\|_{0,\Omega} \Big),
\end{aligned}
$$
which completes the proof.
\end{proof}

Next, by combining the results of Lemmas \ref{lemma1}–\ref{lemma4}, we obtain the following conclusion:
\begin{lemma}[Error Estimate for Discrete Nash Controls]   \label{lemma5}
Let $(\mathbf{u}_1, \mathbf{u}_2)$ be the continuous optimal Nash controls, and let $(\mathbf{u}_{1,h}, \mathbf{u}_{2,h})$ be their discrete counterparts. Denote by $\mathbf{y}_h(\mathbf{u}_1,\mathbf{u}_2)$, $\boldsymbol{\varphi}_{i,h}(\mathbf{u}_1,\mathbf{u}_2)$, $p_h(\mathbf{u}_1,\mathbf{u}_2)$, and $r_{i,h}(\mathbf{u}_1,\mathbf{u}_2)$ the discrete state, adjoint state, state pressure, and adjoint pressure corresponding to the continuous controls $(\mathbf{u}_1,\mathbf{u}_2)$, and let $\mathbf{y}_h$ and $\boldsymbol{\varphi}_{i,h}$ be the state and adjoint corresponding to the discrete controls $(\mathbf{u}_{1,h},\mathbf{u}_{2,h})$. Then, there exists a constant $C>0$, independent of the mesh size $h$, such that
\begin{equation}
\begin{split}  \label{eq:lemma5-bound}
& \|\mathbf{y}_h(\mathbf{u}_1,\mathbf{u}_2) - \mathbf{y}_h\|_{1,\Omega} + \sum_{i=1}^{2} \Big( \alpha_i \|\mathcal{P}_{U_i}^{h}\mathbf{u}_i - \mathbf{u}_{i,h}\|_{0,\Omega} + \|\boldsymbol{\varphi}_{i,h}(\mathbf{u}_1,\mathbf{u}_2) - \boldsymbol{\varphi}_{i,h}\|_{1,\Omega} \Big) \\
& +  \Big( \|p_{h}(\mathbf{u}_1,\mathbf{u}_2) - p_{h}\|_{0;\Omega} + \sum_{i=1}^{2}\|r_{i,h}(\mathbf{u}_1,\mathbf{u}_2) - r_{i,h}\|_{0;\Omega} \Big) \\
\le ~ & C\|\mathbf{y} - \mathbf{y}_{h}(\mathbf{u}_1,\mathbf{u}_2)\|_{0,\Omega} + C \sum_{i=1}^{2} \Big( \|\boldsymbol{\varphi}_i - \boldsymbol{\varphi}_{i,h}(\mathbf{u}_1,\mathbf{u}_2)\|_{0,\Omega} + h_{U_i} \|\mathbf{u}_i - \mathcal{P}_{U_i}^{h} \mathbf{u}_i\|_{0,\Omega} \Big).
\end{split}
\end{equation}
\end{lemma}

\begin{proof}
Using the Lemma~\ref{lemma1}, we immediately obtain the bound
\[    \label{state-error-final}
\|\mathbf{y}_h(\mathbf{u}_1,\mathbf{u}_2) - \mathbf{y}_h\|_{1,\Omega} + \|p_h(\mathbf{u}_1,\mathbf{u}_2) - p_h\|_{0,\Omega} \le C \sum_{i=1,2}\Big( \|\mathcal{P}_{U_i}^{h} \mathbf{u}_i - \mathbf{u}_{i,h}\|_{0,\Omega} + h_{U_i} \|\mathbf{u}_i - \mathcal{P}_{U_i}^{h} \mathbf{u}_i\|_{0,\Omega} \Big).
\]

Using the Lemma~\ref{lemma2}, we immediately obtain the bound 

\[
\left\|\boldsymbol{\varphi}_{i,h}(\mathbf{u}_1,\mathbf{u}_2) - \boldsymbol{\varphi}_{i,h}\right\|_{1;\Omega}
+ \left\|r_{i,h}(\mathbf{u}_1,\mathbf{u}_2) - r_{i,h}\right\|_{0;\Omega}
\leq C \left\|\mathbf{y} - \mathbf{y}_{h}\right\|_{0;\Omega}.
\]

\noindent Accordingly, the final estimate, invoking the discrete control error lemma, is given by: from Lemma~\ref{lemma4} together with Eq.~\eqref{eq:nash-control-error}, the discrete control errors satisfy
\begin{align*}
& \|\mathcal{P}_{U_1}^{h}\mathbf{u}_1 - \mathbf{u}_{1,h}\|_{0,\Omega}
+ \|\mathcal{P}_{U_2}^{h}\mathbf{u}_2 - \mathbf{u}_{2,h}\|_{0,\Omega}  \\
\le ~ & C \Big( \|\mathbf{y}-\mathbf{y}_h(\mathbf{u}_1,\mathbf{u}_2)\|_{0,\Omega} + \sum_{i=1}^{2} \|\boldsymbol{\varphi}_i - \boldsymbol{\varphi}_{i,h}(\mathbf{u}_1,\mathbf{u}_2)\|_{0,\Omega}
+ \sum_{i=1}^{2} h_{U_i} \|\mathbf{u}_i - \mathcal{P}_{U_i}^{h} \mathbf{u}_i\|_{0,\Omega} \Big).
\end{align*}
By incorporating the stability estimates for the discrete state and adjoint errors, we obtain
$$
\begin{aligned}
& \sum_{i=1}^{2} \alpha_i \|\mathcal{P}_{U_i}^{h}\mathbf{u}_i - \mathbf{u}_{i,h}\|_{0,\Omega}
+ \sum_{i=1}^{2} \Big( \|\mathbf{y}_h(\mathbf{u}_1,\mathbf{u}_2) - \mathbf{y}_h\|_{1,\Omega}
+ \|\boldsymbol{\varphi}_{i,h}(\mathbf{u}_1,\mathbf{u}_2) - \boldsymbol{\varphi}_{i,h}\|_{1,\Omega} \Big) \\
& +  \Big( \|p_{h}(\mathbf{u}_1,\mathbf{u}_2) - p_{h}\|_{0,\Omega}
+ \sum_{i=1}^{2}\|r_{i,h}(\mathbf{u}_1,\mathbf{u}_2) - r_{i,h}\|_{0,\Omega} \Big) \\
&\le C \|\mathbf{y} - \mathbf{y}_h(\mathbf{u}_1,\mathbf{u}_2)\|_{0,\Omega}
+ C \sum_{i=1}^{2} \Big( \|\boldsymbol{\varphi}_i - \boldsymbol{\varphi}_{i,h}(\mathbf{u}_1,\mathbf{u}_2)\|_{0,\Omega}
+ h_{U_i} \|\mathbf{u}_i - \mathcal{P}_{U_i}^{h} \mathbf{u}_i\|_{0,\Omega} \Big),
\end{aligned}
$$
Combining all these results, we obtain the desired error estimate.
\end{proof}

\begin{proof}[Proof of the Theorem \ref{therem-1}]
We begin by decomposing the total state error using the triangle inequality:
$$
\|\mathbf{y} - \mathbf{y}_h\|_{0,\Omega} \le \underbrace{\|\mathbf{y} - \mathbf{y}_h(\mathbf{u}_1,\mathbf{u}_2)\|_{0,\Omega}}_{\text{Approximation Error}} + \underbrace{\|\mathbf{y}_h(\mathbf{u}_1,\mathbf{u}_2) - \mathbf{y}_h\|_{0,\Omega}}_{\text{Discrete Error}},
$$
where $\mathbf{y}_h(\mathbf{u}_1,\mathbf{u}_2)$ denotes the discrete state corresponding to the exact controls $(\mathbf{u}_1,\mathbf{u}_2)$. Similarly,
$$
\left\|\boldsymbol{\varphi}_i - \boldsymbol{\varphi}_{i,h}\right\|_{0, \Omega} \leq \left\|\boldsymbol{\varphi}_i - \boldsymbol{\varphi}_{i,h}(\mathbf{u}_1, \mathbf{u}_2)\right\|_{0, \Omega} + \left\|\boldsymbol{\varphi}_{i,h}(\mathbf{u}_1, \mathbf{u}_2) - \boldsymbol{\varphi}_{i,h}\right\|_{0, \Omega}, ~~~~ \text{for} ~~~ i = 1,2.
$$

\smallskip

\noindent The Approximation Error, assuming sufficient regularity, the standard $\mathbf{L}^2$-approximation estimates for the SIPG method yield \cite{BrennerScott1994}

\begin{align}
\|\mathbf{y}-\mathbf{y}_{h}(\mathbf{u}_1,\mathbf{u}_2)\|_{0,\Omega}
&\leq C h^{l+2}\big(\|\mathbf{y}\|_{l+2,\Omega}+\|p\|_{l+1,\Omega}\big), \label{eq:L2_y_est}\\
\|\boldsymbol{\varphi}_i-\boldsymbol{\varphi}_{i,h}(\mathbf{u}_1,\mathbf{u}_2)\|_{0,\Omega}
&\leq C h^{l+2}\big(\|\boldsymbol{\varphi}_i\|_{l+2,\Omega}+\|r_i\|_{l+1,\Omega}\big). \label{eq:L2_phi_est}
\end{align}
Moreover, the $L^2$-projection $\mathcal{P}_{U_i}^{h}$ satisfies
\begin{equation}
    \|\mathbf{u}_i - \mathcal{P}_{U_i}^{h}\mathbf{u}_i\|_{0,\Omega} \leq C h_{U_i}^{k+1}\|\mathbf{u}_i\|_{k+1,\Omega}.
\end{equation}

By inserting \eqref{eq:L2_y_est}–\eqref{eq:L2_phi_est} into the right-hand side of \eqref{eq:lemma5-bound}, we obtain the following transformations for each term:
$$
\begin{aligned}
\|\mathbf{y}-\mathbf{y}_h(\mathbf{u}_1,\mathbf{u}_2)\|_{0,\Omega} &\lesssim h^{\,l+2}, \\
\|\boldsymbol\varphi_i-\boldsymbol\varphi_{i,h}(\mathbf{u}_1,\mathbf{u}_2)\|_{0,\Omega} &\lesssim h^{\,l+2}, \\
h_{U_i}\,\|\mathbf{u}_i-\mathcal P_{U_i}^h\mathbf{u}_i\|_{0,\Omega} &\lesssim h_{U_i}\cdot h_{U_i}^{\,k+1} = h_{U_i}^{\,k+2}.
\end{aligned}
$$
Invoking into \eqref{eq:lemma5-bound} yields
$$
\|\mathbf{y}_h(\mathbf{u}_1,\mathbf{u}_2) - \mathbf{y}_h\|_{0,\Omega} + \sum_{i=1}^{2}\Big(\|\mathcal{P}_{U_i}^{h}\mathbf{u}_i - \mathbf{u}_{i,h}\|_{0,\Omega} + \|\boldsymbol{\varphi}_{i,h}(\mathbf{u}_1,\mathbf{u}_2) - \boldsymbol{\varphi}_{i,h}\|_{0,\Omega} \Big) \le C\big(h^{\,l+2} + h_{U_1}^{\,k+2} + h_{U_2}^{\,k+2}\big),
$$
Therefore there exists a constant $C^{\prime}>0$ (depending on the $\alpha_i$'s and the constants above) such that
\[
\|\mathbf{y}_h(\mathbf{u}_1,\mathbf{u}_2) - \mathbf{y}_h\|_{0,\Omega} + \sum_{i=1}^{2}\Big(\|\boldsymbol{\varphi}_{i,h}(\mathbf{u}_1,\mathbf{u}_2) - \boldsymbol{\varphi}_{i,h}\|_{0,\Omega} + \|\mathcal{P}_{U_i}^{h}\mathbf{u}_i - \mathbf{u}_{i,h}\|_{0,\Omega}\Big) \le C^{\prime} \big(h^{l+2} + h_{U_1}^{k+2} + h_{U_2}^{k+2}\big).
\]
Finally, combine approximation and discrete parts  to obtain the claimed total $L^2$ bound:
\[
\|\mathbf{y}-\mathbf{y}_h\|_{0,\Omega} + \sum_{i=1}^{2}\Big( \|\boldsymbol\varphi_i - \boldsymbol\varphi_{i,h}\|_{0,\Omega} + \|\mathcal{P}_{U_i}^{h}\mathbf{u}_i-\mathbf{u}_{i,h}\|_{0,\Omega}\Big) \le C^{\prime \prime} \big(h^{l+2} + h_{U_1}^{k+2} + h_{U_2}^{k+2}\big),
\]
where $C^{\prime \prime}>0$ independent of $h,h_{U_1},h_{U_2}$. \begin{equation*}
\begin{aligned}
& \left\|\mathbf{y}-\mathbf{y}_{h}\right\|_{1, \Omega} + \left\|\boldsymbol{\varphi}_1 - \boldsymbol{\varphi}_{1,h}\right\|_{1, \Omega} + \left\|\boldsymbol{\varphi}_2 - \boldsymbol{\varphi}_{2,h}\right\|_{1, \Omega} \\
& + \left\|p - p_{h}\right\|_{0, \Omega} + \left\|r_1 - r_{1,h}\right\|_{0, \Omega} + \left\|r_2 - r_{2,h}\right\|_{0, \Omega} 
& \leq C\left(h^{l+1} + h_{U_1}^{k+2} + h_{U_2}^{k+2}\right).
\end{aligned}
\end{equation*}

\end{proof}

\begin{proof}[Proof of the Theorem \ref{theorem-2}]
Using the standard finite element analysis for the Stokes equation \cite{BrennerScott1994} $\mathbf{H}^1$-norm estimates yield 
\begin{align*}
\|\mathbf{y}-\mathbf{y}_{h}(\mathbf{u}_1,\mathbf{u}_2)\|_{1,\Omega}
&\leq C h^{l+1}\big(\|\mathbf{y}\|_{l+2,\Omega}+\|p\|_{l+1,\Omega}\big),\\
\|\boldsymbol{\varphi}_i-\boldsymbol{\varphi}_{i,h}(\mathbf{u}_1,\mathbf{u}_2)\|_{1,\Omega}
&\leq C h^{l+1}\big(\|\boldsymbol{\varphi}_i\|_{l+2,\Omega}+\|r_i\|_{l+1,\Omega}\big). 
\end{align*}
Proceeding similarly to the above proof, we obtain the desired result
\begin{equation*}
\begin{split}
&  \left\|\mathbf{y} - \mathbf{y}_{h}\right\|_{0, \Omega} + \left\|\boldsymbol{\varphi}_1 - \boldsymbol{\varphi}_{1,h}\right\|_{0, \Omega} + \left\|\boldsymbol{\varphi}_2 - \boldsymbol{\varphi}_{2,h}\right\|_{0, \Omega}  \\
& + \left\|\mathcal{P}_{U_1}^{h} \mathbf{u}_1 - \mathbf{u}_{1,h}\right\|_{0, \Omega} + \left\|\mathcal{P}_{U_2}^{h} \mathbf{u}_2 - \mathbf{u}_{2,h}\right\|_{0, \Omega} \leq C \left(h^{l+2} + h_{U_1}^{k+2} + h_{U_2}^{k+2}\right).
\end{split}
\end{equation*}
\end{proof}

\section{Numerical Experiment}\label{sec:4}
This section provides a comprehensive presentation of the iterative algorithm developed for the computation of a Nash quasi-equilibrium corresponding to the cost functional \eqref{eq:cost_functional} and the Stokes state system \eqref{eq:stokes_state_system}. To illustrate the behavior and effectiveness of this algorithm, we also report the results of numerical experiments. The computations for the experiments were performed using the FreeFem++ \cite{hecht2005freefem++}.

\begin{algorithm}[H]
\caption{Fixed-Point-like Method}
\textbf{Input:} Control regions \( \Omega \), viscosity \( \nu > 0 \), velocity space \( V \subset H_0^1(\Omega)^d \) \\
\textbf{Output:} Updated controls \( \mathbf u_1^{n+1}, \mathbf u_2^{n+1} \)

\begin{algorithmic}[1]
\STATE Choose initial controls \( (\mathbf u_1^0, \mathbf u_2^0) \in L^2(\Omega)^d \times L^2(\Omega)^d \) and initial velocity \( \mathbf y^0 \in \mathbf{H} \)
\FOR{each \( n \geq 0 \)}
    \STATE Given \( (\mathbf u_1^n, \mathbf u_2^n) \in L^2(\Omega)^d \times L^2(\Omega)^d \) and \( \mathbf y^n \in \mathbf{H} \), perform the following:
    
    \STATE \textbf{Step 1:} Compute the state \( (\mathbf y^{n+1}, p^{n+1}) \in \mathbf{H} \times L^2_0(\Omega) \) solving:
    \[
    \begin{cases}
    \nabla \cdot \mathbf y^{n+1} = 0 & \text{in } \Omega, \\
    -\nu \Delta \mathbf y^{n+1} + \nabla p^{n+1} = B_1 u_1^n + B_2 u_2^n & \text{in } \Omega, \\
    \mathbf y^{n+1} = 0 & \text{on } \partial \Omega.
    \end{cases}
    \]
    
    \STATE \textbf{Step 2:} For each player \( i = 1, 2 \), compute the adjoint state \( (\boldsymbol \varphi_i^{n+1}, q_i^{n+1}) \in \mathbf{H} \times L^2_0(\Omega) \) solving:
    \[
    \begin{cases}
    \nabla \cdot \boldsymbol\varphi_i^{n+1} = 0 & \text{in } \Omega, \\
    -\nu \Delta \boldsymbol\varphi_i^{n+1} + \nabla q_i^{n+1} = B_i \mathbf u_i^{n+1} & \text{in } \Omega, \\
    \boldsymbol\varphi_i^{n+1} = 0 & \text{on } \partial \Omega.
    \end{cases}
    \]
    
    \STATE \textbf{Step 3:} Update the controls:
    \[
    \mathbf u_i^{n+1} = -\frac{1}{\alpha_i} B_i^* \boldsymbol\varphi_i^{n+1}, \quad \text{for } i = 1,2.
    \]
\ENDFOR
\end{algorithmic}
\end{algorithm}

\begin{algorithm}[H]
\caption{Optimal-Step-Gradient-like Method}
\textbf{Input:} Initial controls \( (\mathbf u_1^0, \mathbf u_2^0) \in L^2(\Omega)^d \times L^2(\Omega)^d \), initial velocity \( \mathbf y^0 \in \mathbf{H} \) \\
\textbf{Output:} Updated controls \( \mathbf u_1^{n+1}, \mathbf u_2^{n+1} \)

\begin{algorithmic}[1]
\STATE Choose initial controls \( (u_1^0, u_2^0) \in L^2(\Omega)^d \times L^2(\Omega)^d \) and initial velocity \( y^0 \in V \)
\FOR{each \( n \geq 0 \)}
    \STATE Given \( (u_1^n, u_2^n) \in L^2(\Omega)^d \times L^2(\Omega)^d \) and \( y^n \in V \), perform the following:
    
    \STATE \textbf{Step 1:} Solve the state system for \( (y^{n+1}, p^{n+1}) \):
    \[
    \nabla \cdot y^{n+1} = 0 \quad \text{in } \Omega,
    \]
    \[
    -\nu \Delta y^{n+1} + \nabla p^{n+1} = B_1u_1^n  + B_2u_2^n  \quad \text{in } \Omega,
    \]
    \[
    y^{n+1} = 0 \quad \text{on } \partial\Omega.
    \]
    
    \STATE \textbf{Step 2:} Solve the adjoint system for each \( i = 1, 2 \), for \( (\varphi_i^{n+1}, q_i^{n+1}) \):
    \[
    \nabla \cdot \boldsymbol\varphi_i^{n+1} = 0 \quad \text{in } \Omega,
    \]
    \[
    -\nu \Delta \boldsymbol\varphi_i^{n+1}  + \nabla q_i^{n+1} = y_i^{n+1} \quad \text{in } \Omega,
    \]
    \[
   \boldsymbol\varphi_i^{n+1} = 0 \quad \text{on } \partial \Omega.
    \]
    
    \STATE \textbf{Step 3:} Define the gradients for each player \( i = 1, 2 \):
    \[
    g_i^{n+1} =  B_i^*\boldsymbol\varphi_i^{n+1}  + \alpha_i u_i^n
    \]
    
    \STATE \textbf{Step 4:} Update the controls:
    \[
    u_i^{n+1} = u_i^n - \rho_i^n g_i^{n+1}, \quad i = 1, 2.
    \]
    
    \STATE \textbf{Step 5:} Choose step sizes \( \rho_1^n \) and \( \rho_2^n \) as:
    \[
    \rho_1^n = \arg\min_{\rho \geq 0} J_1^n(u_1^n - \rho g_1^n, u_2^n),
    \]
    \[
    \rho_2^n = \arg\min_{\rho \geq 0} J_2^n(u_1^n, u_2^n - \rho g_2^n).
    \]
    
    \STATE \textbf{Step 6:} The quadratic functionals \( J_1^n \) and \( J_2^n \) are computed from the linearized system:
    \[
    \nabla \cdot y = 0 \quad \text{in } \Omega,
    \]
    \[
    -\nu \Delta y + \nabla p = B_1u_1  + B_2u_2  \quad \text{in } \Omega,
    \]
    \[
    y = 0 \quad \text{on } \partial \Omega.
    \]
\ENDFOR
\end{algorithmic}
\end{algorithm}

\begin{example}[Multi-domain Optimal Control Problem]
In this example, we illustrate the performance of the finite element approximation for the distributed optimal control problem governed by the stationary incompressible Stokes equations. The continuous optimality system associated with the Nash equilibrium problem \eqref{eq:optimality-system-nash} admits a unique solution $(\mathbf{u}, \mathbf{y}, p, \boldsymbol{\varphi}_i, r_i)$, and as shown in the theoretical analysis, there exists a subsequence that weakly converges to this solution as $h \to 0$. In this numerical discretization, we employ conforming finite element spaces satisfying the inf–sup (Ladyzhenskaya–Babu{\v s}ka–Brezzi) condition.
\vspace{0.8em}

\noindent The computational domain (see Figure~\ref{fig:mesh}) consists of five rectangular subdomains connected through a central horizontal channel, denoted $\Omega_1, \Omega_2, O_1, O_2$, and $\Omega_c$. The discrete velocity and pressure spaces used in the simulation are
$$V_H = [P_2(\Omega)]^2, \qquad Q_H = P_1(\Omega),$$
which correspond to the Taylor--Hood pair $(\mathbf{P}_2, P_1)$, ensuring
stability of the mixed formulation.

\noindent \noindent
The desired state $\mathbf{y}_{1d}$ in the control subdomain $O_1$ is obtained from the stream function $\psi$ solving
\[
\begin{cases}
- \Delta \psi = 1 & \text{in } O_1,\\
\psi = 0 & \text{on } \partial O_1,
\end{cases}
\qquad
\mathbf{y}_{1d} = (\partial_y \psi, -\partial_x \psi),
\]
while in $O_2$, the desired control $\mathbf{u}_{2d}$ is set to zero. The initial desired field $\mathbf{y}_{1d}$ is depicted in Figure~\ref{fig:u1d}.

\noindent
The physical parameters in this simulation include the Reynolds number $\mathrm{Re}$ and two constants $a = 1.99$ and $\mu = 2 - a$, representing the regularization and viscosity coefficients, respectively. Numerical simulations are carried out for $\mathrm{Re}=240$, $720$, and $1200$ to study the effect of convective dominance on the optimal state and control.

The objective functional minimized in this problem is given by
\[
J(\mathbf{u}_1, \mathbf{u}_2)
= \frac{1}{2}\|\mathbf{y} - \mathbf{y}_d\|^2_{\mathbf{L}^2(\Omega)}
+ \frac{\alpha_1}{2}\|\mathbf{u}_1\|^2_{\mathbf{L}^2(\mathbf{U}_1)}
+ \frac{\alpha_2}{2}\|\mathbf{u}_2\|^2_{\mathbf{L}^2(\mathbf{U}_2)}.
\]
This configuration serves as a benchmark for the distributed optimal control of incompressible flows over composite domains.

\vspace{0.5em}
\noindent
The resulting computational mesh and flow fields are illustrated below.
As shown in Figures~\ref{fig:Re240}–\ref{fig:Re1200}, the solution exhibits
laminar behavior for $\mathrm{Re}=240$, while vortical structures emerge
for higher Reynolds numbers, consistent with theoretical expectations and
the weak convergence framework as $h \to 0$.

\begin{figure}[H]
    \centering
    \includegraphics[width=0.95\textwidth]{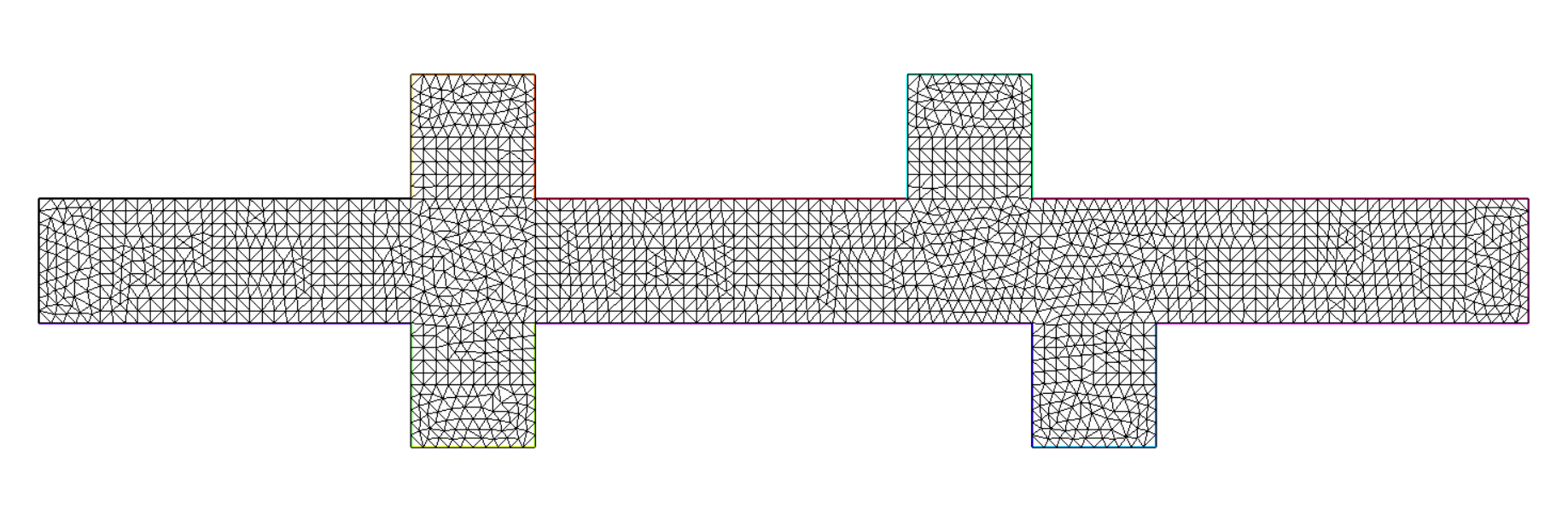}
    \caption{Triangular finite element mesh for the multi-domain configuration.}
    \label{fig:mesh}
\end{figure}

\begin{figure}[H]
    \centering
    \includegraphics[width=0.95\textwidth]{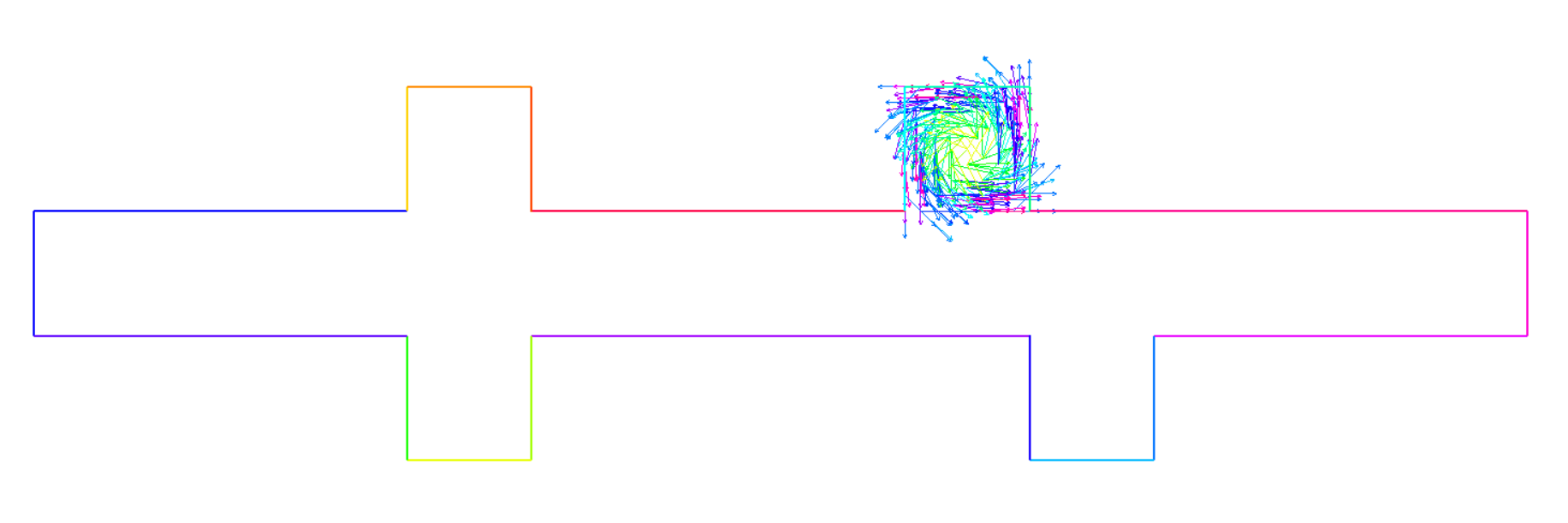}
    \caption{Desired initial velocity field $\mathbf{y}_{1d}$ in subdomain $O_1$.}
    \label{fig:u1d}
\end{figure}

\begin{figure}[H]
    \centering
    \includegraphics[width=0.95\textwidth]{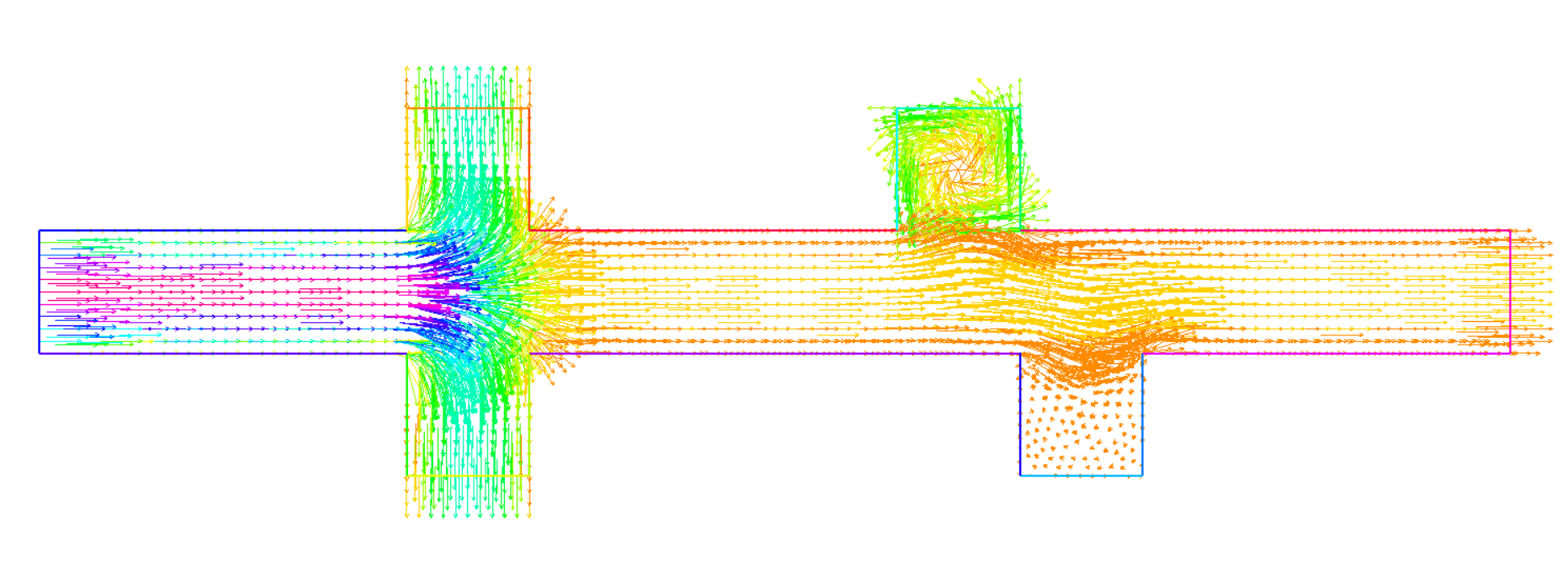}
    \caption{Velocity field at $\mathrm{Re}=240$ (laminar regime).}
    \label{fig:Re240}
\end{figure}

\begin{figure}[H]
    \centering
    \includegraphics[width=0.95\textwidth]{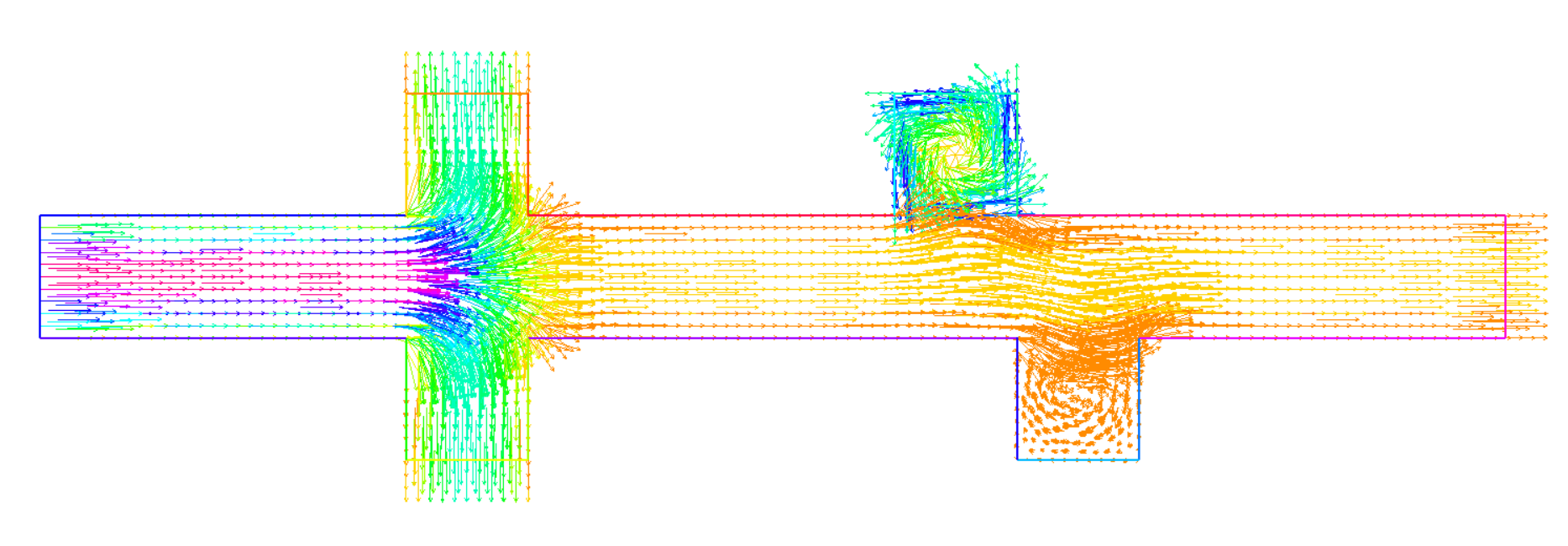}
    \caption{Velocity field at $\mathrm{Re}=720$ showing onset of recirculation zones.}
    \label{fig:Re720}
\end{figure}

\begin{figure}[H]
    \centering
    \includegraphics[width=0.95\textwidth]{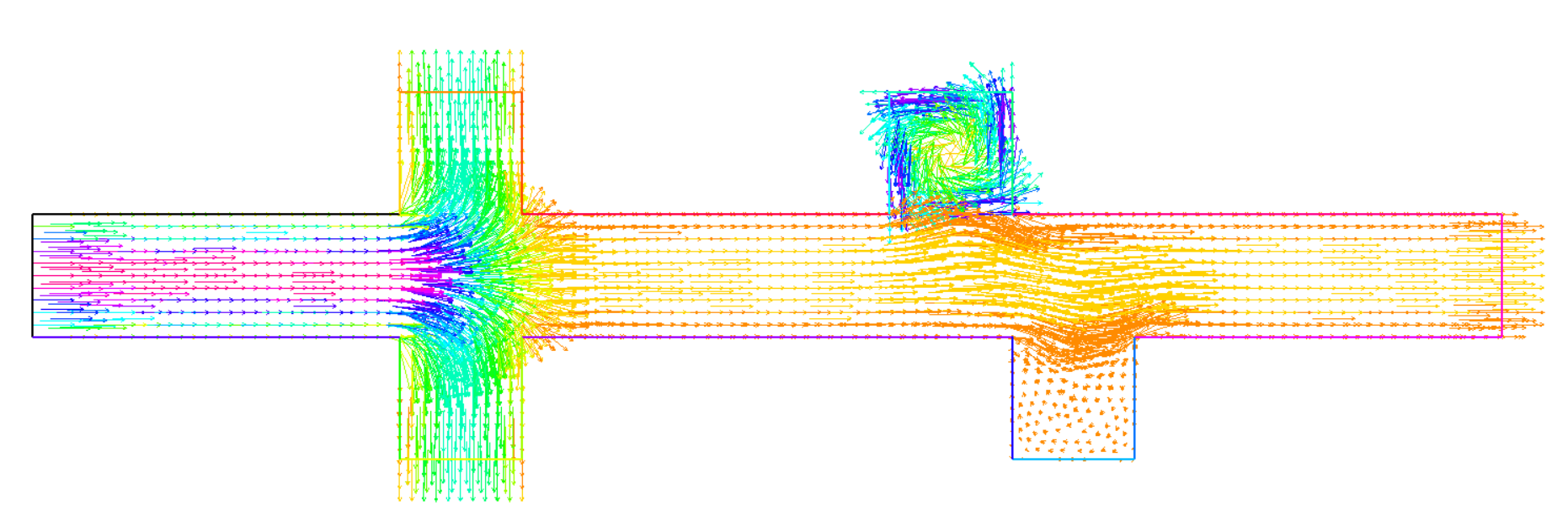}
    \caption{Velocity field at $\mathrm{Re}=1200$ exhibiting strong vortex formation.}
    \label{fig:Re1200}
\end{figure}
\end{example}

\section{Conclusion}\label{sec:5}
In this work, we have investigated the Nash equilibrium associated with a bi-objective optimal control problem governed by the Stokes equations, considering both cooperative and non-cooperative settings. A rigorous theoretical setting was developed to analyze the existence, uniqueness, and analytical characterization of the Nash equilibrium. The corresponding finite element approximation was formulated, and the optimality conditions for both the continuous and discrete problems were derived and thoroughly studied. Furthermore, we established \textit{a priori} error estimates for the finite element discretization, providing a solid theoretical foundation for the numerical approximation of the equilibrium solutions. The numerical experiments, implemented using the Finite Element Method (FEM), confirmed the accuracy and efficiency of the proposed algorithms. Overall, the study contributes to the understanding of multi-objective Nash strategies in fluid dynamics control problems and offers practical computational techniques for their numerical realization. 
In future work, we aim to extend this study to a Nash–Stackelberg methodology for the Stokes equations, where one player leads, and others follow, and further explore its application to fractional PDE models involving hierarchical leader–follower interactions.

\section*{Declarations}

\textbf{Conflicts of Interest:} The authors declare that they have no conflicts of interest.

\bibliographystyle{plain} 
\bibliography{references} 

\end{document}